\documentclass[a4paper, 12pt]{amsart}
\usepackage[mathscr]{eucal}
\usepackage{amssymb}
\usepackage{latexsym}
\usepackage{amsthm}
\usepackage{color}
\theoremstyle{plain}

\newcommand{\be}{\begin{equation}}
	\newcommand{\ee}{\end{equation}}
\newcommand{\ba}{\begin{eqnarray}}
	\newcommand{\ea}{\end{eqnarray}}
\newcommand{\ban}{\begin{eqnarray*}}
	\newcommand{\ean}{\end{eqnarray*}}

\def\XXint#1#2#3{{\setbox0=\hbox{$#1{#2#3}{\int}$}
		\vcenter{\hbox{$#2#3$}}\kern-.5\wd0}}

\newtheorem{theorem}{Theorem}[section]
\newtheorem{corollary}{Corollary}[section]

\newtheorem{remark}{Remark}[section]
\newtheorem{lemma}{Lemma}[section]
\newtheorem{proposition}{Proposition}[section]

\makeatletter
\@addtoreset{equation}{section}

\newcommand{\inj}{\mbox{inj}}
\newcommand{\vol}{\mathrm{Vol}}

\newcommand{\Ric}{\mathrm{Ric}}
\newcommand{\Rm}{\mathrm{Rm}}
\newcommand{\Id}{\mathrm{Id}}
\newcommand{\II}{\mathrm{II}}

\newcommand{\tr}{\mbox{tr}}

\newtheorem{theo}{Theorem}[section]

\begin{document}
	\newtheorem{defn}[theo]{Definition}
	\newtheorem{ques}[theo]{Question}
	\newtheorem{lem}[theo]{Lemma}
	\newtheorem{prop}[theo]{Proposition}
	\newtheorem{coro}[theo]{Corollary}
	\newtheorem{ex}[theo]{Example}
	\newtheorem{note}[theo]{Note}
	\newtheorem{conj}[theo]{Conjecture}
	\title[Constant scalar curvature]{Rigidity of Five-dimensional Shrinking Gradient Ricci Solitons with Constant Scalar Curvature II}
	\author{Fengjiang Li}
	\address[Fengjiang Li]
	{Mathematical Science Research Center, Chongqing University of Technology, Chongqing 400054, China}
	\email{fengjiangli@cqut.edu.cn}
	
	\author{Jianyu Ou}
	\address[Jianyu Ou]
	{School of Mathematical Sciences, Xiamen University, Xiamen 361005, China}
	\email{oujianyu@xmu.edu.cn}
	\author{Yuanyuan Qu}
	\address[Yuanyuan Qu]{School of Mathematical Sciences, Shanghai Key Laboratory of PMMP, East China Normal University, Shanghai 200241,
		China}
	\email{52285500012@stu.ecnu.edu.cn}
	
	\author{Guoqiang Wu}
	\address[Guoqiang Wu]
	{School of Science, Zhejiang Sci-Tech University, Hangzhou 310018, China}
	\email{gqwu@zstu.edu.cn}
	
	\subjclass[2020]{Primary 53E20; Secondary 53C21.}
	\keywords{Ricci soliton, constant scalar curvature, bounded curvature, focal set, rigidity.}
	\date{}


%

\begin{abstract} We investigate 5-dimensional shrinking gradient Ricci soliton $(M^5, g,f)$ satisfying $Ric+\nabla^2f= \frac{1}{2} g$  with constant scalar curvature. 
Fern\'{a}ndez-L\'{o}pez and Garc\'\i a-R\'\i o \cite{FR16} proved that the possible values of  scalar curvature are  $\{0,1,\frac{3}{2}, 2, \frac{5}{2}\}$. If $R=0$, the soliton is isometric to $\mathbb{R}^5$; if $R=2$, the soliton is a finite quotient of ${N}^{4}\times\mathbb{R}$ \cite{FR16}, where ${N}^4$ is a four-dimensional Einstein manifold;  if $R=\frac{5}{2}$, the soliton is Einstein \cite{Petersen-Wylie2}. The most difficult cases are $R=\frac{3}{2}$ or $R=1$.
In \cite{Li-Ou-Qu-Wu} we proved that it is a finite quotient of $\mathbb{R}^2\times \mathbb{S}^3$ if $R=\frac{3}{2}$.
In this paper, we prove that every five-dimensional shrinking gradient
Ricci soliton with $R=1$ is isometric to a finite quotient of
$\mathbb{R}^3\times\mathbb{S}^2(\sqrt2)$. Thus every complete connected five-dimensional shrinking gradient Ricci soliton with constant scalar curvature is rigid, so Cao's conjecture
holds in dimension five.
\end{abstract}
\maketitle
\section{Introduction}

Let  $(M, g, f)$ be an $n$-dimensional complete gradient Ricci soliton with the potential function $f$ satisfying
\begin{align}\label{soliton}
	\text{Ric}+\nabla^2f=\lambda g
\end{align}
for some constant $\lambda$, where $\text{Ric}$ is the Ricci tensor of $g$ and $\nabla^2f$ denotes the Hessian of the potential function $f$.
The Ricci soliton is said to be shrinking, steady, or expanding accordingly as $\lambda$ is positive, zero, or negative, respectively. By rescalling the metric $g$ by a positive constant, we can assume $\lambda\in \{\frac{1}{2}, 0, -\frac{1}{2}\}$.
For the shrinking case we take $\lambda=\frac12$ and write the soliton equation as
\begin{equation}\label{soliton'}
\Ric+\nabla^2 f=\frac12g.
\end{equation}
\smallskip	

A gradient Ricci soliton is a self-similar solution to the Ricci flow which flows by diffeomorphisms and homotheties. The study of solitons has become increasingly important in both the study of the Ricci flow introduced by Hamilton \cite{Hamilton} and metric measure theory. Solitons play a direct role as singularity dilations in the Ricci flow proof of uniformization.  Due to the work of Perelman \cite{Perelman2}, Ni-Wallach \cite{Ni-Wallach}, Cao-Chen-Zhu \cite{Cao-Chen-Zhu}, the classification of three-dimensional shrinking gradient Ricci soliton is complete. For more work on the classification of gradient Ricci soliton under various curvature condition, see \cite{Brendle1,Brendle2,Cao-Chen-Zhu,Cao-Chen,Cao-Chen2,Cao-Wang-Zhang,Chen-Wang,Eminenti-LaNave-Mantegazza,Kotschwar,Munteanu-Wang7,Munteanu-Wang6,Munteanu-Wang1,Munteanu-Wang3,Munteanu-Wang2,Munteanu-Wang5,Munteanu-Wang4,Naber,Petersen-Wylie2,Pigola-Rimoldi-Setti,Wu-Zhang,Wu-Wu-Wylie,Zhang}.

\smallskip	
A gradient Ricci soliton $(M, g)$ is said to be rigid if it is isometric to a quotient ${N} \times \mathbb{R}^k$, the product soliton of an Einstein manifold ${N}$ of positive scalar curvature with the Gaussian soliton $\mathbb{R}^k$.
Conversely, for the complete shrinking case, Prof. Huai-Dong Cao raised the following

\smallskip	
\noindent {\bf Conjecture}:
Let $(M^n, g, f)$, $n\geq 4$, be a complete $n$-dimensional gradient shrinking Ricci soliton. If $(M, g)$ has constant scalar curvature, then it must be rigid,
i.e., a finite quotient of ${N}^k\times \mathbb{R}^{n-k}$ for some Einstein manifold ${N}$ of positive scalar curvature.

\smallskip	

Petersen and Wylie \cite{Petersen-Wylie2} proved that a complete gradient Ricci soliton is rigid if and only if it has
constant scalar curvature and is radially flat, that is, the sectional curvature $K(\cdot, \nabla f)=0$. Fern\'{a}ndez-L\'{o}pez and Garc\'\i a-R\'\i o \cite{FR16} obtained that the soliton is rigid if and only if the Ricci curvature has constant rank. They also derived the following results for complete $n$-dimensional gradient Ricci solitons \eqref{soliton} with constant scalar curvature $R$: (i) The possible values of $R$ are $\{0, \frac{1}{2}, \cdots, \frac{n-1}{2}, \frac{n}{2}\}$. (ii) If $R$ takes the value $\frac{n-1}{2}$,  then the soliton must be rigid. (iii) In the shrinking case, there is  no any complete gradient shrinking Ricci soliton  with $R=\frac{1}{2}$.
(iv) Any $n$-dimensional gradient shrinking Ricci soliton with constant scalar curvature $R=\frac{n-2}{2}$ has non-negative Ricci curvature.

\smallskip

In dimension $n=4$, Cheng and Zhou \cite{Cheng-Zhou} confirmed Cao's conjecture.
Very recently, the authors gave a simple proof of Cheng-Zhou's result in \cite{Ou-Qu-Wu}, see also \cite{Wu-Wu}.
In dimension $n=5$, the possible value of the scalar curvature is $\{0, 1, \frac{3}{2}, 2, \frac{5}{2}\}$. From above discussion, it is easy to see that $(M^5, g, f)$ is isometric to $\mathbb{R}^5$ if $R=0$, a finite quotient of $\mathbb{R}\times {N}^4$ (${N}^4$ is a four-dimensional Einstein manifold) if $R=2$, a finite quotient of ${N}^5$ (${N}^5$ is a five-dimensional Einstein manifold) if $R=\frac{5}{2}$. The most difficult case is $R=1$ or $R=\frac{3}{2}$. In \cite{Li-Ou-Qu-Wu}, we proved that five-dimensional shrinking gradient Ricci soliton with constant scalar curvature $\frac{3}{2}$ is isometric to a finite quotient  of $\mathbb{R}^2\times \mathbb{S}^3$. In this paper we settle the remaining case $R=1$. Our main result is as follows.

\begin{theo}\label{maintheorem}
Let $(M^5,g,f)$ be a complete connected shrinking gradient Ricci soliton
satisfying \eqref{soliton'} with $R=1$. Then $(M,g)$ is
isometric to a finite quotient of
\[
\mathbb S^2(\sqrt2)\times\mathbb R^3.
\]
\end{theo}

Combining Theorem~\ref{maintheorem} with the known scalar-curvature
cases \cite{FR16,Li-Ou-Qu-Wu,Petersen-Wylie2} gives the following
consequence.
\begin{corollary}\label{classification}
Every complete connected five-dimensional shrinking gradient Ricci
soliton with constant scalar curvature is rigid. Thus Cao's conjecture
holds in dimension five.
\end{corollary}

The proof has three steps. In Section~\ref{section:curvature}, area
comparison and homology rule out Ricci-flat ALE blow-up limits; hence
$|\Rm|$ is bounded. In Section~\ref{section:focal}, the critical set
$M_-=f^{-1}(0)$ is totally geodesic, and Gauss--Bonnet gives
$|M_-|=8\pi$ on the universal cover. In Section~\ref{n:sec:rigidity},
a weighted integral identity yields $\nabla\Ric=0$ and the splitting.

\section{Preliminaries}\label{section2}
We recall the soliton identities and the geometry of the level sets of
$f$. We then list the results used for the curvature and rigidity
arguments. Write $\Delta_f=\Delta-\nabla_{\nabla f}$.

The standard soliton identities \cite{Cao-survey,Petersen-Wylie2} are
\begin{align}
\nabla R&=2\Ric(\nabla f,\cdot),\label{second bianchi}\\
R+\Delta f&=\frac n2,\label{R}\\
R+|\nabla f|^2&=f,\label{f he tidu}\\
\Delta_fR&=R-2|\Ric|^2,\label{scalar-bochner}\\
\Delta_f\Ric_{ij}&=\Ric_{ij}-2R_{ikjl}\Ric_{kl},\label{elliptic equation}\\
\Delta_f|\Ric|^2&=2|\nabla\Ric|^2+2|\Ric|^2
 -4R_{ikjl}\Ric_{ij}\Ric_{kl}.\label{ric-bochner}
\end{align}
When $n=5$ and $R=1$, an additive normalization of $f$ gives
\begin{equation}\label{basicR1}
|\nabla f|^2=f,\qquad \Delta f=\frac32,\qquad
\Ric(\nabla f,\cdot)=0,\qquad |\Ric|^2=\frac12,
\qquad |\nabla^2f|^2=\frac34.
\end{equation}
In particular $M$ is noncompact, since $\Delta f=3/2$.

The identities above determine the level geometry once $f$ is
normalized. We first use the growth estimate of Cao--Zhou
\cite{Cao-Zhou}: there are constants
$C_1,C_2$ such that
\begin{equation}\label{potential estimate}
\left(\max\left\{0,\frac12d(x,p)-C_1\right\}\right)^2\le f(x)
\le\left(\frac12d(x,p)+C_2\right)^2,
\end{equation}
where $p$ is a minimum point of $f$. Thus $f$ is proper.

By \eqref{basicR1}, $f$ is isoparametric. Cheng--Zhou
\cite{Cheng-Zhou} describe its focal set and regular levels; see also
\cite{Wang-isoparametric}.

\begin{prop}[Cheng--Zhou]\label{levelset}
Let $(M^5,g,f)$ satisfy \eqref{soliton'} and $R=1$, with $f$
normalized by \eqref{basicR1}. Then:

\smallskip
\noindent{\rm (i)} $M_-=f^{-1}(0)$ is a compact connected minimal
surface.

\smallskip
\noindent{\rm (ii)} For every $x\in M$,
\begin{equation}\label{focal-distance}
f(x)=\frac14d(x,M_-)^2.
\end{equation}

\smallskip
\noindent{\rm (iii)} At $p\in M_-$, the eigenspaces of $\nabla^2f$
are $T_pM_-$ and $\nu_pM_-$, with eigenvalues $0$ and $1/2$.
Consequently,
\begin{equation}\label{focalRic}
\Ric|_{TM_-}=\frac12g|_{TM_-},\qquad
\Ric|_{\nu M_-}=0.
\end{equation}

\smallskip
\noindent{\rm (iv)} Put $D_s=\{f\le s\}$ and
$\Sigma(s)=f^{-1}(s)$, $s>0$. Then
\begin{equation}\label{focal-volume}
\vol(D_s)=\frac{32\pi}{3}s^{3/2}|M_-|,
\qquad \vol(\Sigma(s))=16\pi s|M_-|.
\end{equation}
\end{prop}

We shall work on the universal cover. It is a finite cover by Wylie's
theorem.
\begin{theo}[Wylie \cite{Wylie}]\label{finite-pi1}
The fundamental group of a complete shrinking Ricci soliton is finite.
\end{theo}

To identify the regular levels with the normal sphere bundle of
$M_-$, put $r=2\sqrt f$.
By \eqref{basicR1},
\begin{equation}\label{radial-identities}
|\nabla r|=1,\qquad \nabla_{\nabla r}\nabla r=0,
\qquad \Delta r=\frac2r.
\end{equation}

\begin{lemma}\label{sec1_3:lem:normal}
The normal exponential map $\exp^\perp:\nu M_-\to M$ is a
diffeomorphism. Each $\Sigma(s)$ is the unit normal sphere bundle of
$M_-$. On the universal cover, $M_-\cong\mathbb S^2$ and
$\pi_1(\Sigma(s))=0$.
\end{lemma}
\begin{proof}
By \eqref{radial-identities}, gradient lines of $r$ are unit-speed
geodesics. Following $-\nabla r$ for length $r(x)$ reaches $M_-$.
Also $|\nabla r|=1$ gives $r(x)\le d(x,M_-)$. Thus these geodesics
minimize globally. There are no focal points. Two normal
representations of a regular point have the same terminal velocity
$\nabla r$ and coincide. Hence $\exp^\perp$ is a diffeomorphism.
The zero section is a deformation retract. On the universal cover,
$M_-$ is a compact simply connected surface, hence a sphere. The
homotopy sequence of the normal sphere bundle gives
$\pi_1(\Sigma(s))=0$.
\end{proof}

The curvature argument uses noncollapsing and the topology of a
four-dimensional Ricci-flat limit. We record these inputs next.
\begin{theo}[Li--Wang \cite{Li-Wang1}]\label{li-wang}
For a shrinker with $R=1$, there is $\kappa>0$ such that
\begin{equation}\label{noncollapse}
\vol(B_g(x,\rho))\ge\kappa\rho^5,
\qquad x\in M,\quad 0<\rho\le1.
\end{equation}
\end{theo}

\begin{theo}[Cheeger--Naber \cite{Cheeger-Naber}]\label{ale-input}
A complete Ricci-flat four-manifold with Euclidean volume growth is ALE.
\end{theo}

For a nonflat ALE limit we also use Zhang's obstruction.
\begin{lem}[Zhang {\cite[Section~6, Lemmas~6.3 and~6.5]{Zhang-ALE}}]\label{zhang-input}
Let $(N^4,h)$ be a complete nonflat Ricci-flat ALE manifold. Let
$W\subset N$ be a compact core with connected boundary
$Y=\mathbb S^3/\Gamma$. If $b_2(W)=0$, then
$H_1(Y;\mathbb Z)\not\cong G\oplus G$ for any finite abelian group $G$.
\end{lem}

Finally, the ancient limits in Section~\ref{section:focal} use Naber's
theorem.
\begin{theo}[Naber {\cite[Theorem~2.1 and Proposition~4.1]{Naber}}]\label{naber-input}
Let $(M,g,f)$ be a complete noncompact shrinking gradient Ricci
soliton satisfying \eqref{soliton'} with bounded curvature.
For every sequence $q_j\to\infty$, a subsequence of its pointed
canonical flows converges smoothly to a complete ancient flow
which splits off a line.

Let $(N,g(t))$, $t<0$, be a complete ancient Ricci flow which is
$\kappa$-noncollapsed on all scales and satisfies
\[
 |\Rm(g(t))|\le\frac C{-t}.
\]
For a fixed $q\in N$ and $\tau_j\downarrow0$ or $\tau_j\to\infty$,
the pointed flows $(N,\tau_j^{-1}g(\tau_j t),q)$ have a subsequence
converging smoothly to a gradient shrinking soliton,
normalized at $t=-1$.
\end{theo}

\section{Boundedness of curvature}\label{section:curvature}
We first bound the curvature. The proof combines area comparison on
$\Sigma(s)$ with the topology of a possible ALE blow-up limit.
\begin{theorem}\label{b:sec4_6:curvaturebound}
Let $(M^5,g,f)$ be a complete shrinking gradient Ricci soliton
satisfying \eqref{soliton'} and $R=1$. Then
\[
\sup_M|\Rm|<\infty.
\]
\end{theorem}

\subsection{Area comparison and homology}

We begin with the negative gradient flow. The soliton equation controls
the area of its tangent two-planes.

\begin{proposition}\label{sec1_3:prop:area}
Fix $a>0$. For every $s\ge a$, there exists a diffeomorphism
\[
F_s:\Sigma(s)\longrightarrow\Sigma(a)
\]
such that $\|\Lambda^2dF_s\|\le1$, where the metrics on the level sets are
induced by $g$.
\end{proposition}
\begin{proof}
For an orthonormal basis $e_1,e_2$ of a two-plane $P$,
\[
\Ric(e_1,e_1)+\Ric(e_2,e_2)
\le\sqrt2|\Ric|=1.
\]
By \eqref{soliton'},
\[
\operatorname{tr}_P\nabla^2f
=1-\Ric(e_1,e_1)-\Ric(e_2,e_2)\ge0.
\]
Let $\Phi_t$ be the flow of $-\nabla f$. The growth estimate for $f$ implies
that $|\nabla f|$ has at most linear growth, so this vector field is
complete. Moreover,
\[
\frac{d}{dt}f(\Phi_t(x))=-|\nabla f|^2(\Phi_t(x))=-f(\Phi_t(x)),
\]
and hence $f(\Phi_t(x))=e^{-t}f(x)$. Define
\begin{equation}\label{sec1_3:eq:3.2}
F_s=\left.\Phi_{\log(s/a)}\right|_{\Sigma(s)}.
\end{equation}
This is a diffeomorphism onto $\Sigma(a)$.

For a two-dimensional plane $P$, let $J_2(t,P)$ denote the area Jacobian
of $d\Phi_t|_P$. Since $\mathcal L_{-\nabla f}g=-2\nabla^2f$, we obtain
\[
\frac{d}{dt}\log J_2(t,P)
=-\operatorname{tr}_{d\Phi_t(P)}\nabla^2f\le0
\]
by the preceding estimate. Thus $J_2(t,P)\le1$ for $t\ge0$, which proves the
proposition.
\end{proof}

The area bound gives a lower mass bound for nontrivial two-cycles.
We use coefficients in $\mathbb Z$ or $\mathbb Z_m$ and write their mass
as $\mathbf M_g$.

\begin{proposition}\label{sec1_3:prop:cycles}
Fix $a>0$ and $A=\mathbb Z$ or $A=\mathbb Z_m$. There is a constant
$c_A>0$ such that for every $s\ge a$ and every two-dimensional Lipschitz
cycle $Z$ in $\Sigma(s)$,
\begin{equation}\label{sec1_3:eq:3.3}
[Z]\ne0\text{ in }H_2(\Sigma(s);A)
\quad\Longrightarrow\quad\mathbf M_g(Z)\ge c_A.
\end{equation}
\end{proposition}
\begin{proof}
The map $F_s$ in Proposition~\ref{sec1_3:prop:area} induces an isomorphism on homology. The
area formula gives
\[
\mathbf M_g(F_{s\#}Z)\le\mathbf M_g(Z).
\]
By the deformation theorem \cite{Federer-Fleming,White},
there is $c_A>0$ such that every two-cycle of mass below $c_A$ in the
fixed compact manifold $\Sigma(a)$ is null-homologous. Indeed, it
deforms $F_{s\#}Z$ to a simplicial cycle of mass at most
$C\mathbf M_g(F_{s\#}Z)$. A nonzero simplicial cycle has positive
minimum mass. This proves the proposition.
\end{proof}

To apply Zhang's obstruction \cite{Zhang-ALE}, we show that the two
homology maps in \eqref{sec1_3:eq:3.4} vanish. The required duality holds
for every integer $m\ge2$.
\begin{lemma}\label{sec1_3:lem:duality}
Let $W$ be a compact oriented four-manifold with boundary $Y$, and let
$m\ge2$. The intersection pairing
\[
H_2(W,Y;\mathbb Z_m)\times H_2(W;\mathbb Z_m)
\longrightarrow\mathbb Z_m
\]
is nondegenerate. If $\iota:W\hookrightarrow X$ is an embedding into a
closed oriented four-manifold and
\[
\iota_*:H_2(W;\mathbb Z_m)\longrightarrow H_2(X;\mathbb Z_m)
\]
is zero, then the restriction to relative homology
\[
j_m:H_2(X;\mathbb Z_m)\longrightarrow H_2(W,Y;\mathbb Z_m)
\]
is also zero.
\end{lemma}
\begin{proof}
Put $A=\mathbb Z_m$. A finite cellular chain complex for $W$ with
coefficients in $A$ consists of finite free $A$-modules. The character map
$A\to\mathbb Q/\mathbb Z$, $\bar k\mapsto k/m+\mathbb Z$, identifies
\[
\operatorname{Hom}_A(C_*,A)
=\operatorname{Hom}_{\mathbb Z}(C_*,\mathbb Q/\mathbb Z).
\]
Since $\mathbb Q/\mathbb Z$ is an injective abelian group, evaluation gives
\[
H^2(W;A)\cong\operatorname{Hom}_A(H_2(W;A),A).
\]
Characters separate points of a finite abelian group. Combining this fact
with Poincar\'e--Lefschetz duality proves nondegeneracy of the intersection
pairing.

For $\alpha\in H_2(X;A)$ and $\beta\in H_2(W;A)$, naturality of duality
gives
\[
j_m(\alpha)\mathbin{\cdot_W}\beta
=\alpha\mathbin{\cdot_X}\iota_*(\beta).
\]
If $\iota_*=0$, the right-hand side vanishes for every $\beta$.
Nondegeneracy implies $j_m(\alpha)=0$.
\end{proof}

\begin{proposition}\label{sec1_3:prop:obstruction}
Let $X$ be a closed oriented simply connected four-manifold, and let
$W\subset X$ be a compact connected smooth domain with connected boundary
$Y=S^3/\Gamma$. Put
\[
m=\max\{2,|H_1(Y;\mathbb Z)|\}.
\]
Suppose both maps
\begin{equation}\label{sec1_3:eq:3.4}
H_2(W;\mathbb Z)\longrightarrow H_2(X;\mathbb Z),\qquad
H_2(W;\mathbb Z_m)\longrightarrow H_2(X;\mathbb Z_m)
\end{equation}
are zero. Then $\operatorname{Int}W$ does not admit a complete nonflat
Ricci flat ALE metric for which $W$ is a compact core with boundary at
infinity $Y$.
\end{proposition}
\begin{proof}
All groups without specified coefficients in this proof have integer
coefficients. Write $V=X\setminus W$. The space form $Y$ has finite first
homology and $H_2(Y)=0$. The Mayer--Vietoris sequence therefore contains
\[
0\longrightarrow H_2(W)\oplus H_2(V)\longrightarrow H_2(X).
\]
The first map in \eqref{sec1_3:eq:3.4} is both injective and zero. Hence
\begin{equation}\label{sec1_3:eq:3.5}
H_2(W)=0.
\end{equation}

Consider the next part of the same sequence:
\begin{equation}\label{sec1_3:eq:3.6}
H_2(X)\xrightarrow{\delta}H_1(Y)
\longrightarrow H_1(W)\oplus H_1(V)
\longrightarrow H_1(X)=0.
\end{equation}
We claim that $\delta=0$. With coefficients in $\mathbb Z_m$, the
connecting map factors, up to sign, through
\[
H_2(X;\mathbb Z_m)\xrightarrow{j_m}H_2(W,Y;\mathbb Z_m)
\xrightarrow{\partial}H_1(Y;\mathbb Z_m).
\]
Lemma~\ref{sec1_3:lem:duality} and the second map in \eqref{sec1_3:eq:3.4} give $j_m=0$, so $\delta_m=0$. The
coefficient homomorphism
\[
H_1(Y;\mathbb Z)\longrightarrow H_1(Y;\mathbb Z_m)
\]
is injective: the universal coefficient theorem identifies its target with
$H_1(Y;\mathbb Z)\otimes\mathbb Z_m$, and multiplication by $m$
annihilates $H_1(Y;\mathbb Z)$. Naturality now gives $\delta=0$, as claimed.
Consequently, \eqref{sec1_3:eq:3.6} yields an isomorphism
\begin{equation}\label{sec1_3:eq:3.7}
H_1(Y)\cong H_1(W)\oplus H_1(V).
\end{equation}

Let $G=H_1(W)$. This group is finite, since $H_1(Y)$ is finite. The
inclusion $H_1(Y)\to H_1(W)$ is surjective, and the boundary $Y$ is
connected. The sequence of the pair $(W,Y)$ then gives $H_1(W,Y)=0$ and
\[
0\longrightarrow H_2(W,Y)\longrightarrow H_1(Y)
\longrightarrow G\longrightarrow0.
\]
By \eqref{sec1_3:eq:3.5}, Poincar\'e--Lefschetz duality, and the universal coefficient
theorem,
\[
H_2(W,Y)\cong H^2(W;\mathbb Z)
\cong\operatorname{Ext}(G,\mathbb Z)\cong G.
\]
Under \eqref{sec1_3:eq:3.7}, the homomorphism $H_1(Y)\to G$ is the projection onto the
first summand, up to sign. Thus it has a section, and the preceding short
exact sequence splits. We conclude
\begin{equation}\label{sec1_3:eq:3.8}
H_1(Y)\cong G\oplus G.
\end{equation}
If $\operatorname{Int}W$ admitted the asserted ALE metric, \eqref{sec1_3:eq:3.5} and \eqref{sec1_3:eq:3.8}
would contradict the ALE obstruction recalled in Section~\ref{section2}.
\end{proof}

\subsection{Curvature blow-up and ALE limits}
Suppose now that $|\Rm|$ is unbounded. We construct an ALE limit and
embed its compact core in a regular level of $f$.

\begin{proposition}\label{ale-blowup}
Suppose $(M^5,g,f)$ satisfies \eqref{soliton'} and $R=1$, and $|\Rm|$ is unbounded.
There exist points $p_i$ tending to infinity and numbers $Q_i=|\Rm|_g(p_i)\to\infty$
such that
\begin{equation}\label{b:sec4_6:blowup}
(M,Q_i g,p_i)\longrightarrow (\mathbb R\times N^4,dt^2+h,(0,p_\infty))
\end{equation}
smoothly on compact sets. Here $(N,h)$ is a complete nonflat Ricci flat ALE manifold.
Moreover, for every compact smooth domain $W\subset N$, there are embeddings
\begin{equation}\label{b:sec4_6:embedding}
\iota_i:W\longrightarrow\Sigma(f(p_i))
\end{equation}
such that $\iota_i^*(Q_i g)$ converges to $h|_W$ uniformly.
\end{proposition}

\begin{proof}
Let $\phi_t$, $t<0$, be generated by
\[
\frac{d}{dt}\phi_t=\frac{\nabla f}{-t}\circ\phi_t,
\qquad \phi_{-1}=\Id.
\]
Then
\begin{equation}\label{b:sec4_6:canonical}
g(t)=(-t)\phi_t^*g
\end{equation}
is a complete ancient solution of the Ricci flow. On $[-2,-1]$,
\begin{equation}\label{b:sec4_6:ricbounds}
|\Ric|_{g(t)}=\frac{1}{\sqrt2(-t)}\le1,
\qquad R_{g(t)}=\frac{1}{-t}\le1.
\end{equation}
In particular, $g(t)$ and $g=g(-1)$ are uniformly equivalent on this time interval.

\emph{Step 1. Local parabolic point picking.}
Choose $x_i$ such that $K_i=|\Rm|_g(x_i)\ge i^6$.
Starting at $(x_i,-1)$, perform the following procedure. If the current point is
$(y,t)$ and $b=|\Rm|_{g(t)}(y)$, and there is a point $(z,u)$ in
\[
B_g(y,i b^{-1/2})\times[t-i^2b^{-1},t]
\]
with $|\Rm|_{g(u)}(z)>4b$, replace $(y,t)$ by $(z,u)$.
At each replacement the curvature increases by a factor greater than four.
The total spatial displacement in the fixed metric $g$ is at most
$2iK_i^{-1/2}$, and the total backward time displacement is at most
$\frac43 i^2K_i^{-1}$.

For sufficiently large $i$, all selected points remain in the compact cylinder
\[
\overline B_g(x_i,2iK_i^{-1/2})\times[-1-2i^2K_i^{-1},-1]
\subset M\times[-2,-1].
\]
Smoothness of the flow on this cylinder implies that the procedure terminates.
Denote the final point by $(y_i,t_i)$ and put $b_i=|\Rm|_{g(t_i)}(y_i)$.
We have $b_i\ge K_i$, $t_i\to-1$, and
\begin{equation}\label{b:sec4_6:pointpicking}
|\Rm|_{g(t)}\le4b_i
\quad\text{on }B_g(y_i,i b_i^{-1/2})\times[t_i-i^2b_i^{-1},t_i].
\end{equation}

Consider the rescaled flows
\[
\widehat g_i(\tau)=b_i g(t_i+b_i^{-1}\tau),\qquad \tau\le0.
\]
Uniform equivalence of the metrics on $[-2,-1]$ shows that \eqref{b:sec4_6:pointpicking} contains
parabolic neighborhoods whose spatial radii and backward time lengths tend to
infinity in the rescaled metrics. On these neighborhoods $|\Rm|\le4$.

Noncollapsing \eqref{noncollapse} gives a uniform lower volume bound
for rescaled unit balls. The curvature bound gives an injectivity radius
bound. Shi estimates \cite{Shi} and Hamilton compactness
\cite{Hamilton-compactness} then yield a complete ancient limit. The
controlled balls exhaust its time-zero slice. By \eqref{b:sec4_6:ricbounds},
the limit is static and Ricci flat, with $|\Rm|=1$ at the base point.

\emph{Step 2. Pullback and splitting.} Set
\begin{equation}\label{b:sec4_6:pullback}
p_i=\phi_{t_i}(y_i),\qquad Q_i=(-t_i)b_i.
\end{equation}
The time-zero slice $(M,\widehat g_i(0),y_i)$ is isometric to $(M,Q_i g,p_i)$,
and $Q_i=|\Rm|_g(p_i)$. The points $x_i$, and hence $y_i$, leave every compact
set. Since
\[
f(\phi_t(x))=\frac{f(x)}{-t},
\]
we have $s_i:=f(p_i)\to\infty$.

Define
\begin{equation}\label{b:sec4_6:auxpotential}
f_i=\frac{\sqrt{Q_i}}{\sqrt{s_i}}(f-s_i).
\end{equation}
With respect to $Q_i g$, the identities in \eqref{basicR1} imply
\begin{equation}\label{b:sec4_6:auxderivatives}
|\nabla f_i|^2=\frac{f}{s_i},\qquad
|\nabla^2f_i|=\frac{\sqrt3}{2\sqrt{Q_i s_i}}.
\end{equation}
Furthermore, $\sqrt f$ is $\frac12$-Lipschitz in the metric $g$.
On every fixed radius ball in $(M,Q_i g,p_i)$, we therefore have $f/s_i\to1$
uniformly. It follows that $f_i$ converges, after passing to a subsequence, in $C^1$
on compact sets to a function $f_\infty$ with
\[
f_\infty(p_\infty)=0,\qquad |\nabla f_\infty|=1,
\qquad \nabla^2 f_\infty=0.
\]
The last identity first holds weakly, then smoothly. The unit parallel
gradient gives the global splitting
\[
(M_\infty,g_\infty)=(\mathbb R\times N,dt^2+h),\qquad f_\infty=t.
\]
The factor $N$ is complete, Ricci flat, and nonflat. On the oriented
universal cover, the line gradient also orients $N$.

For fixed $L>0$, pass \eqref{noncollapse} to the limit:
\[
\vol_{g_\infty}(B_{g_\infty}(p_\infty,L))\ge\kappa L^5.
\]
On the product, this ball is contained in $[-L,L]\times B_h(p_\infty,L)$. Thus
\[
\vol_h(B_h(p_\infty,L))\ge\frac\kappa2L^4.
\]
By \cite[Corollary 8.85]{Cheeger-Naber}, $N$ is ALE.

\emph{Step 3. Embedding the compact core into a level set.}
Let $W\subset N$ be a compact smooth domain. Choose a slightly larger compact
neighborhood of $W$ and use the convergence embeddings in \eqref{b:sec4_6:blowup} on its
product with a small interval $[-\varepsilon,\varepsilon]$. The pullbacks of $f_i$
converge to the coordinate $t$ in $C^1$. Their $t$-derivatives are therefore positive
for large $i$. The implicit function theorem represents their zero sets over $W$
as graphs $t=v_i(z)$, where $v_i\to0$ in $C^1$. The images of these graphs are
the embeddings \eqref{b:sec4_6:embedding}, because $f_i=0$ is precisely $f=s_i$. Smooth
convergence of the ambient metrics and $C^1$ convergence of the graphs imply
\[
\iota_i^*(Q_i g)\longrightarrow h|_W
\]
uniformly. This completes the proof.
\end{proof}

We now prove Theorem~\ref{b:sec4_6:curvaturebound}. The area bound
forces both homology maps of the embedded ALE core to vanish.

\begin{proof}[Proof of Theorem~\ref{b:sec4_6:curvaturebound}]
Pass to the universal cover and retain the notation $(M,g,f)$.
Suppose the curvature is unbounded, and apply Proposition~\ref{ale-blowup}. Take a compact
core $W$ of the nonflat Ricci flat ALE factor $N$, with $Y=\partial W=S^3/\Gamma$,
and fix
\[
m=\max\{2,|H_1(Y;\mathbb Z)|\}.
\]
Write $s_i=f(p_i)$. The maps $\iota_i:W\to\Sigma(s_i)$ identify
$W$ with a compact smooth domain in $\Sigma(s_i)$.
By Lemma~\ref{sec1_3:lem:normal}, $\Sigma(s_i)$ is a closed
oriented simply connected four-manifold.

For $A=\mathbb Z,\mathbb Z_m$, choose finite Lipschitz two-cycles
generating $H_2(W;A)$. For each such $Z$, Proposition~\ref{ale-blowup}
gives
\[
\mathbf M_g(\iota_{i\#}Z)=Q_i^{-1}\mathbf M_{Q_i g}(\iota_{i\#}Z)
\le\frac{C_Z}{Q_i}\longrightarrow0.
\]
Fix $a>0$. We have $s_i\ge a$ for large $i$. Proposition~\ref{sec1_3:prop:cycles} now implies that all
these cycles have zero homology class in $\Sigma(s_i)$ for sufficiently large $i$.
Consequently, both homomorphisms
\[
H_2(W;\mathbb Z)\longrightarrow H_2(\Sigma(s_i);\mathbb Z),\qquad
H_2(W;\mathbb Z_m)\longrightarrow H_2(\Sigma(s_i);\mathbb Z_m)
\]
are zero. One index bound works for both coefficient groups.
Proposition~\ref{sec1_3:prop:obstruction} gives a contradiction.
Thus $|\Rm|$ is bounded.
\end{proof}

\section{Total geodesy of the focal set}\label{section:focal}
In this section we prove that $M_-$ is totally geodesic. First we
identify the limits at infinity. Then we compare the curvature of
$\Sigma(s)$ with its cylindrical limit and use Gauss--Bonnet.

\subsection{Ancient factors and cylindrical limits}
The curvature bound gives smooth pointed limits. Their four-dimensional
factors are ancient flows. We classify them first.

\begin{theorem}\label{b:sec4_6:ancientfactor}
Let $(N^4,g(t))$, $t<0$, be a complete ancient Ricci flow satisfying
\begin{equation}\label{b:sec4_6:ancientbounds}
R_{g(t)}=(-t)^{-1},\qquad |\Rm(g(t))|\le C(-t)^{-1},
\end{equation}
and suppose that it is $\kappa$-noncollapsed on all scales. Then it is a smooth finite
quotient of $S^2(\sqrt{-2t})\times\mathbb R^2$.
\end{theorem}

Since scalar curvature is spatially constant, its evolution equation
$\partial_tR=\Delta R+2|\Ric|^2$ gives
\begin{equation}\label{b:sec4_6:ricnorm}
|\Ric_{g(t)}|^2=\frac1{2(-t)^2}=\frac12R_{g(t)}^2.
\end{equation}
For $\tau_i\downarrow0$ or $\tau_i\to\infty$, consider
$\tau_i^{-1}g(\tau_i t)$. Naber's endpoint theorem \cite{Naber}
gives shrinker limits. Both have $R=1$ at $t=-1$ and are noncompact:
the compact case would give $\Delta f=1$. By Cheng--Zhou
\cite{Cheng-Zhou}, each limit is a finite quotient of
$S^2(\sqrt2)\times\mathbb R^2$. Denote the quotients at
$\tau_i\downarrow0$ and $\tau_i\to\infty$ by
\begin{equation}\label{b:sec4_6:endpoints}
\bigl(S^2(\sqrt2)\times\mathbb R^2\bigr)/\Gamma_0,
\qquad
\bigl(S^2(\sqrt2)\times\mathbb R^2\bigr)/\Gamma_\infty,
\end{equation}
respectively. Here $\Gamma_0$ and $\Gamma_\infty$ are finite groups
of isometries of $S^2(\sqrt2)\times\mathbb R^2$; each acts freely,
and its quotient is the corresponding limit.

Such a smooth quotient has $|\Gamma|\in\{1,2\}$. Indeed, the action
on $S^2$ is free, so $2/|\Gamma|=\chi(S^2/\Gamma)$ is an integer.
The normalized weighted volume is inversely proportional to
$|\Gamma|$. Reduced-volume monotonicity \cite{Perelman-entropy} gives
$|\Gamma_0|\le |\Gamma_\infty|$.
If $|\Gamma_0|=|\Gamma_\infty|$, equality makes the flow a shrinker.
The classification of Cheng--Zhou \cite{Cheng-Zhou} gives the conclusion.
It remains to consider
\begin{equation}\label{b:sec4_6:unequaldensity}
|\Gamma_0|=1,\qquad |\Gamma_\infty|=2.
\end{equation}

Assume \eqref{b:sec4_6:unequaldensity}, and let
$X=(S^2(\sqrt2)\times\mathbb R^2)/\Gamma_\infty$. The spherical action
of the nonidentity element is antipodal. Let $g_{\rm cyl}$ be the
quotient product metric and write its potential as
\[
f_{\rm cyl}=|z|^2/4+c.
\]
The critical set $Z=(S^2\times\{0\})/\Gamma_\infty$ is $\mathbb{RP}^2$, and radial
contraction retracts $X$ onto $Z$. Choose a smooth embedded loop
$\overline\gamma\subset Z$ generating $\pi_1(X)=\mathbb Z_2$.

For the shifted flow $g(t-1)$, $t\le0$, let
$(4\pi|t|)^{-2}e^{-f(t)}$ be the conjugate heat kernel based at a fixed
point at time zero. Let $\phi_t$ be its negative-gradient gauge and set
\begin{equation}\label{b:sec4_6:normalizedgauge}
\widehat g(s)=e^s\phi_{-e^{-s}}^*g(-1-e^{-s}),\qquad
\widehat f(s)=\phi_{-e^{-s}}^*f(-e^{-s}).
\end{equation}
Thus
\begin{equation}\label{b:sec4_6:gaugeevolution}
\partial_s\widehat g=\widehat g-2\Ric_{\widehat g}-2\nabla^2 \widehat f.
\end{equation}
The backward limits in \eqref{b:sec4_6:endpoints} are asymptotic
shrinkers. Chan--Ma--Zhang \cite{Chan-Ma-Zhang} identify them with
tangent flows at infinity and give the entropy bound. Thus the
backward tangent flow is the quotient cylinder
$X=(S^2(\sqrt2)\times\mathbb R^2)/\Gamma_\infty$.
The pair $(\widehat g,\widehat f)$ in
\eqref{b:sec4_6:normalizedgauge} is the modified flow for the same
conjugate heat kernel. Fang--Li \cite[Theorem 8.19]{Fang-Li}
therefore give compatible embeddings $\psi_{-j}:\Omega_j\to N$.
Their model radius is $2\sqrt{|f_{\rm cyl}|}$; its square differs
from $|z|^2$ by a bounded constant for large $|z|$.
After decreasing $\varepsilon$, we may take
\[
\Omega_j=\{|z|\le\sqrt{(8-\varepsilon)\log j}\},\qquad
\psi_{-j-1}=\psi_{-j}\quad\text{on }\Omega_j,
\]
and, pointwise on $\Omega_j$ for $s\le-j$,
\begin{equation}\label{b:sec4_6:compatiblecharts}
\begin{split}
&\sum_{\ell=0}^{k}\Bigl(
 |\nabla^\ell(\psi_{-j}^*\widehat g(s)-g_{\rm cyl})|
 +|\nabla^\ell(\psi_{-j}^*\widehat f(s)-f_{\rm cyl})|\Bigr)\\
&\hspace{2em}\le Ce^{(f_{\rm cyl}-\min f_{\rm cyl})/2}|s|^{-1+\varepsilon}.
\end{split}
\end{equation}
All derivatives and norms use $g_{\rm cyl}$; $k$ is fixed and
$\varepsilon$ is sufficiently small. Compatibility defines
an embedding $\Psi:X\to N$ on the union of these domains.

Choose a compact tubular neighborhood $K$ of $Z$ with smooth boundary.
The model gradient $z\partial_z/2$ points strictly outward on $\partial K$.
By \eqref{b:sec4_6:compatiblecharts}, the same holds for
$\Psi^*\nabla_{\widehat g(s)}\widehat f(s)$ when $s\le s_0$, for a sufficiently negative $s_0$.
Let $F_{s_0,s}$ be the flow of this gradient, with $F_{s_0,s_0}=\Id$.
Equivalently,
\[
F_{s_0,s}=\phi_{-e^{-s}}^{-1}\circ\phi_{-e^{-s_0}}.
\]
Backward invariance of $K$ gives loops
\[
\gamma_s=F_{s_0,s}(\Psi(\overline\gamma))\subset\Psi(K),\qquad s\le s_0.
\]
Their inverse images under $\Psi$ remain homotopic to $\overline\gamma$ in $K$,
and hence are nontrivial in $\pi_1(X)$. Smooth convergence on $K$ also gives a fixed
point $q\in\Psi(Z)$ and a constant $R_0$ such that
\begin{equation}\label{b:sec4_6:loopball}
\gamma_s\subset B_{\widehat g(s)}(q,R_0),\qquad s\le s_0.
\end{equation}

We use area comparison to show that $N$ is not simply connected
in this case.
For a two-plane $\Pi$, set $\rho=R_{g(t)}$.
By \eqref{b:sec4_6:ricnorm}, Cauchy--Schwarz gives
\[
\frac{\rho^2}{2}\ge
 \frac{(\operatorname{tr}_{\Pi}\Ric)^2}{2}
 +\frac{(\rho-\operatorname{tr}_{\Pi}\Ric)^2}{2}.
\]
Hence $0\le\operatorname{tr}_{\Pi}\Ric\le\rho$.

Pull \eqref{b:sec4_6:gaugeevolution} back by $F_{s_0,s}$ and put
$\widetilde g(s)=F_{s_0,s}^*\widehat g(s)$. The Hessian term cancels, yielding
\[
\partial_s\widetilde g=\widetilde g-2\Ric_{\widetilde g},\qquad
R_{\widetilde g}=\frac1{1+e^s},\qquad
|\Ric_{\widetilde g}|^2=\frac12(1+e^s)^{-2}.
\]
For a fixed immersed surface $u:D\to N$, area variation and the preceding estimate give
\begin{equation}\label{b:sec4_6:areavariation}
\begin{split}
\frac{d}{ds}\operatorname{Area}_{\widetilde g(s)}(u)
&=\int_D(1-\tr_{Tu(D)}\Ric_{\widetilde g})\,dA\\
&\ge(1-(1+e^s)^{-1})\operatorname{Area}_{\widetilde g(s)}(u)\ge0.
\end{split}
\end{equation}
If $N$ were simply connected, $\gamma_{s_0}$ would bound a smooth disk $u_0$.
Relative transversality permits an immersed interior. Therefore
$u_s=F_{s_0,s}\circ u_0$ spans $\gamma_s$ and
\begin{equation}\label{b:sec4_6:fillingarea}
\operatorname{Area}_{\widehat g(s)}(u_s)
\le\operatorname{Area}_{\widehat g(s_0)}(u_0)=:A_0,\qquad s\le s_0.
\end{equation}
The Type-I bound and noncollapsing imply uniform bounds
$|\Rm(\widehat g(s))|\le K_0$ and $\inj_{\widehat g(s)}\ge i_0>0$; for the latter use
\cite{Cheeger-Gromov-Taylor} on a sufficiently small fixed radius. We record the filling estimate
needed to use \eqref{b:sec4_6:fillingarea}.

\begin{lemma}\label{ancient-filling-area}
In a complete manifold with $|\Rm|\le K_0$ and injectivity radius at least $i_0$,
a smooth embedded loop admitting a spanning disk of area at most $A_0$ admits
an area-minimizing parametrized disk whose image lies within distance
$R_1=R_1(A_0,K_0,i_0)$ of the loop.
\end{lemma}
\begin{proof}
A complete finite-dimensional Riemannian manifold is proper. The Plateau theorem
\cite{Lytchak-Wenger} provides an area-minimizing disk, with the usual harmonic regularity in
a smooth target. Away from its boundary, its associated varifold is stationary.
Choose $r_0<\frac1{10}\min\{i_0,K_0^{-1/2}\}$. The monotonicity formula gives
area at least $c_0r_0^2$ in every radius-$r_0$ ball centered on the image and disjoint
from the boundary curve. If an image point has distance $L$ from that curve, a
path in the image joining them crosses the distance levels $2r_0,5r_0,8r_0,\ldots$.
The corresponding balls are disjoint, since distance to the curve is one-Lipschitz.
Hence
\[
A_0\ge c_0r_0^2\max\left\{0,1+\left\lfloor\frac{L-2r_0}{3r_0}\right\rfloor\right\}.
\]
This gives $L\le3r_0(2+A_0/(c_0r_0^2))$, as required.
\end{proof}

Combining \eqref{b:sec4_6:loopball} with Lemma~\ref{ancient-filling-area}, choose minimizing disks whose images lie in
$B_{\widehat g(s)}(q,R_0+R_1)$. Choose $a>0$ such that $a<\sqrt{8-\varepsilon}$ and
$\varepsilon+a^2/8<1$, and put $\Omega'_j=\{|z|\le a\sqrt{\log j}\}$.
At $s=-j$, the right side of \eqref{b:sec4_6:compatiblecharts} is at most
$Cj^{-1+\varepsilon+a^2/8}$ on $\Omega'_j$.
Thus the metrics there are uniformly bilipschitz for large $j$. A path from $q$ to
the complement of $\Psi(\Omega'_j)$ must first cross its boundary; its length before
that crossing is at least $c\sqrt{\log j}$. Consequently
\[
B_{\widehat g(-j)}(q,c\sqrt{\log j})\subset\Psi(\Omega'_j).
\]
For large $j$, the minimizing disk is contained in this chart. Its inverse image is
a disk in $X$ spanning a nontrivial loop, a contradiction. Hence $N$ is not simply
connected in the case \eqref{b:sec4_6:unequaldensity}.

Bamler's tangent-at-infinity theorem \cite{Bamler} now gives a surjection
$\pi_1(X)\to\pi_1(N)$. Since the source is $\mathbb Z_2$ and the target is nontrivial,
$\pi_1(N)=\mathbb Z_2$. The universal-cover statement of the same theorem identifies
the fundamental group of the regular part of the backward limit of the universal
cover with the kernel of this surjection. That kernel is trivial. Endpoint compactness
makes the limit smooth, so the backward endpoint of the universal cover is
$S^2(\sqrt2)\times\mathbb R^2$.

The forward endpoint is also $S^2(\sqrt2)\times\mathbb R^2$. The degree-two covering
maps in \eqref{b:sec4_6:unequaldensity} limit to a complete local
isometry onto the simply connected forward endpoint, hence to an
isometry. The universal-cover flow therefore has equal endpoint
densities. Naber's equality case makes it the product shrinking flow.
Hence $(N,g(t))$ is a smooth finite quotient of
$S^2(\sqrt{-2t})\times\mathbb R^2$. This proves
Theorem~\ref{b:sec4_6:ancientfactor}.

We apply this classification to pointed limits of the five-dimensional
shrinker.
\begin{proposition}\label{cylindrical-compactness}
For every sequence $q_j\to\infty$ in the shrinker of Theorem~\ref{maintheorem}, a subsequence of
the pointed canonical flows converges smoothly to a finite quotient of
$\mathbb R\times S^2(\sqrt{-2t})\times\mathbb R^2$.
\end{proposition}
\begin{proof}
Theorem~\ref{b:sec4_6:curvaturebound}, noncollapsing, and Shi estimates give pointed smooth compactness on
compact subintervals of $t<0$. On the time-$-1$ slice, the functions
\[
\frac{f-f(q_j)}{\sqrt{f(q_j)}}
\]
have gradients tending to unit length and Hessians tending to zero on compact
pointed sets. The limit therefore splits off a global line. The same splitting
persists under the flow, and the four-dimensional factor has scalar curvature
$(-t)^{-1}$, the Type-I bound, and noncollapsing. Apply Theorem~\ref{b:sec4_6:ancientfactor}.
\end{proof}

\begin{remark}
The radial direction is a zero Ricci eigenvector. It need not correspond to the
smallest ordered eigenvalue. Cylindrical compactness nevertheless implies that
the ordered Ricci eigenvalues tend uniformly at infinity to $(0,0,0,1/2,1/2)$.
\end{remark}

\subsection{Focal geometry and regular levels}
We work on the universal cover. Each $\Sigma(s)$ is a closed oriented
four-manifold and an $\mathbb S^2$-bundle over $M_-\cong\mathbb S^2$.
We use the convention $R_{ijij}>0$ on positively curved planes.
We first relate the geometry of $M_-$ to its area.

\begin{lemma}\label{focal-geometry}
Along $M_-$, with second fundamental form $\II$, one has
\begin{equation}\label{b:sec7_9:7.1}
 |\nabla\Ric|^2=\frac34|\II|^2,\qquad K_{M_-}=\frac12+\frac14|\II|^2.
\end{equation}
In particular,
\begin{equation}\label{b:sec7_9:7.2}
 |M_-|=8\pi-\frac12\int_{M_-}|\II|^2\,dA,
 \qquad \vol(\Sigma(s))=16\pi|M_-|s.
\end{equation}
\end{lemma}
\begin{proof}
At a focal point, the eigenvalues of $\nabla^2 f$ are $0$ on $TM_-$ and $1/2$ on $\nu M_-$. Since $df=0$, the third derivative of $f$ is fully symmetric. Differentiating $|df|^2=f$ three times in a Hessian eigenbasis gives
\[
 2(h_i+h_j+h_k)f_{ijk}=f_{ijk},\qquad h_i\in\{0,1/2\}.
\]
Thus only components with one normal and two tangent indices survive. For tangent vectors $X,Y$ and a normal vector $\xi$, tangentially differentiating $\Ric(Y,\xi)=0$ gives
\[
 (\nabla_X\Ric)(Y,\xi)=\frac12\langle\II(X,Y),\xi\rangle.
\]
There are three possible positions for the normal index, proving the first identity in \eqref{b:sec7_9:7.1}. Differentiating the constant trace of $\Ric$ in a normal direction also gives $\tr_{TM_-}\II=0$, so $M_-$ is minimal.

At the focal point, the Ricci Bochner identity \eqref{ric-bochner} and \eqref{focalRic} imply
\[
 0=2|\nabla\Ric|^2+1-2K_M(TM_-).
\]
The Gauss equation for a minimal surface gives $K_{M_-}=K_M(TM_-)-|\II|^2/2$, proving the second identity. Gauss--Bonnet on the sphere now yields
\[
4\pi=\frac{|M_-|}{2}+\frac14\int_{M_-}|\II|^2\,dA.
\]

Finally, $\Delta r=2/r$ implies that the normal-coordinate volume density is exactly $r^2$ times its limiting density along $M_-$. Integrating over unit normal spheres, of area $4\pi$, gives $\vol\{r=\text{constant}\}=4\pi r^2|M_-|$. Since $r=2\sqrt{s}$, this is the second identity in \eqref{b:sec7_9:7.2}.
\end{proof}

The flow of $\nabla f/f$ identifies all regular levels with a fixed level.
For large $s$ we use
\begin{equation}\label{b:sec7_9:7.3}
 s=e^\tau+\frac12,
\end{equation}
and write $g_s$ for the pullback of $g|_{T\Sigma(s)}$
to that fixed level. In formulas on $\Sigma(s)$, the ambient tensor
$\nabla_{\nabla f}\Ric_g$ is restricted to tangent vectors and
pulled back by the same flow. We raise an index when using Ricci
tensors as endomorphisms.

\begin{lemma}
The metrics $g_s$ satisfy
\begin{equation}\label{b:sec7_9:7.4}
 \partial_\tau g_s=g_s-2\Ric_{g_s}-\frac{2\nabla_{\nabla f}\Ric_g}{s}.
\end{equation}
Moreover,
\begin{equation}\label{b:sec7_9:7.5}
 \begin{aligned}
 R_{g_s}&=1+\frac1{2s},\\
 |\Ric_{g_s}|^2
 &=\frac12+\frac{|\nabla_{\nabla f}\Ric_g|^2+1/8}{s^2}.
 \end{aligned}
\end{equation}
For every fixed nonnegative integer $m$,
\begin{equation}\label{b:sec7_9:7.6}
 \|\nabla^m \nabla_{\nabla f}\Ric_g\|_\infty\longrightarrow0,
\end{equation}
and
\begin{equation}\label{b:sec7_9:7.7}
 \partial_\tau|\Ric_{g_s}|^2=O(s^{-3/2}),\qquad \nabla^m(|\Ric_{g_s}|^2-1/2)=O(s^{-2}).
\end{equation}
\end{lemma}
\begin{proof}
The second fundamental form of $\Sigma(s)$ is
$(g_s-2\Ric|_{T\Sigma})/(2\sqrt{s})$. The Gauss and radial
Riccati equations give
\begin{equation}\label{b:sec7_9:7.8}
 \Ric_{g_s}=\Ric|_{T\Sigma}
     -\frac{\nabla_{\nabla f}\Ric_g}{s}
     +\frac{g_s-2\Ric|_{T\Sigma}}{4s}.
\end{equation}
Here the radial curvature term follows by differentiating the second
fundamental form along $\nabla f/\sqrt{s}$. The level velocity is
$(e^\tau/s)\nabla f$. Hence
\[
 \partial_\tau g_s=\frac{e^\tau}{s}
       (g_s-2\Ric|_{T\Sigma}),\qquad
 g_s-2\Ric_{g_s}=\frac{e^\tau}{s}
       (g_s-2\Ric|_{T\Sigma})+\frac{2\nabla_{\nabla f}\Ric_g}s,
\]
which proves \eqref{b:sec7_9:7.4}.

By \eqref{basicR1} and differentiation along $\nabla f$,
\[
 \tr(\Ric|_{T\Sigma})=1,\qquad
 |\Ric|_{T\Sigma}|^2=\frac12,\qquad
 \tr \nabla_{\nabla f}\Ric_g=\langle\Ric|_{T\Sigma},\nabla_{\nabla f}\Ric_g\rangle=0.
\]
It follows that
\[
 |g_s-2\Ric|_{T\Sigma}|^2=2,\qquad
 \langle g_s-2\Ric|_{T\Sigma},\nabla_{\nabla f}\Ric_g\rangle
 =\langle\Ric|_{T\Sigma},g_s-2\Ric|_{T\Sigma}\rangle=0.
\]
Tracing and squaring \eqref{b:sec7_9:7.8} give
\eqref{b:sec7_9:7.5}.

The curvature bound and Shi estimates bound all ambient curvature derivatives. The drift equation gives
\begin{equation}\label{b:sec7_9:7.9}
 \nabla_{\nabla f}\Ric=\Delta\Ric-\Ric+2\Rm_g(\Ric).
\end{equation}
Thus this tensor and all its spatial derivatives are bounded. Restriction to levels preserves the bounds, since the level second fundamental form and its fixed-order spatial derivatives are $O(s^{-1/2})$. On every pointed cylindrical limit the right side of \eqref{b:sec7_9:7.9}, with all derivatives, vanishes. Proposition~\ref{cylindrical-compactness} therefore makes these bounds $o(1)$ uniformly at infinity. This proves \eqref{b:sec7_9:7.6}.

The material derivative of $\nabla_{\nabla f}\Ric_g$ is $O(\sqrt{s})$: its ambient covariant radial derivative is bounded, the velocity has size $O(\sqrt{s})$, and the tensor pullback terms are bounded.
Since $|\partial_\tau g_s|\le C$, differentiating the inverse metrics
in \eqref{b:sec7_9:7.5} contributes $O(s^{-2})$. Thus
\[
 \begin{aligned}
 \partial_\tau|\Ric_{g_s}|^2
 &=2s^{-2}
 \langle\nabla_{\nabla f}\Ric_g,
        \partial_\tau(\nabla_{\nabla f}\Ric_g)\rangle\\
 &\quad-2s^{-3}(s-1/2)
       (|\nabla_{\nabla f}\Ric_g|^2+1/8)+O(s^{-2})\\
 &=O(s^{-3/2}).
 \end{aligned}
\]
Spatial differentiation of the same identity proves the remaining estimates in \eqref{b:sec7_9:7.7}.
\end{proof}

We combine the level identities with Gauss--Bonnet.
\begin{lemma}\label{level-curvature}
On the universal cover, every regular level satisfies
\begin{equation}\label{b:sec7_9:7.10}
 \begin{aligned}
 8\pi\int_{M_-}|\II|^2\,dA
 =\int_{\Sigma(s)}
 \Bigl(&|\Rm_{g_s}|^2-1-\frac1{4s^2}\\
 &-\frac{4|\nabla_{\nabla f}\Ric_g|^2}{s^2}\Bigr)dV_{g_s}.
 \end{aligned}
\end{equation}
\end{lemma}
\begin{proof}
The Euler characteristic of the sphere bundle is $\chi(\Sigma(s))=\chi(\mathbb S^2)^2=4$, regardless of whether the bundle is trivial. In full tensor norms the four-dimensional Gauss--Bonnet formula reads
\[
 32\pi^2\chi(\Sigma(s))=\int_{\Sigma(s)}\bigl(|\Rm_{g_s}|^2-4|\Ric_{g_s}|^2+R_{g_s}^2\bigr)dV_{g_s}.
\]
Using \eqref{b:sec7_9:7.5}, the integrand is $|\Rm_{g_s}|^2-1+s^{-1}-(4|\nabla_{\nabla f}\Ric_g|^2+1/4)s^{-2}$. By \eqref{b:sec7_9:7.2}, the right side of \eqref{b:sec7_9:7.10} equals
\[
128\pi^2-\frac{\vol(\Sigma(s))}{s}
=128\pi^2-16\pi|M_-|
=8\pi\int_{M_-}|\II|^2\,dA.
\]
\end{proof}

\subsection{Decay of the intrinsic curvature defect}
Cylindrical compactness gives convergence, but no rate. For
\eqref{b:sec7_9:7.10} we need decay faster than $s^{-1/2}$.
We normalize the intrinsic curvature defect and its first derivative.

For any symmetric two-tensor, we use the quadratic expression
\[
 \mathcal M(\Ric_{g_s})_{ijkl}
 =2\bigl((\Ric_{g_s})_{ik}(\Ric_{g_s})_{jl}
         -(\Ric_{g_s})_{il}(\Ric_{g_s})_{jk}\bigr).
\]
The same formula defines $\mathcal M$ on other symmetric tensors.
Its derivative at a symmetric tensor $E$ is
\[
 D\mathcal M_E(H)_{ijkl}
 =2\bigl(H_{ik}E_{jl}+E_{ik}H_{jl}
          -H_{il}E_{jk}-E_{il}H_{jk}\bigr).
\]
On the product cylinder, $\mathcal M(\Ric)$ is its curvature tensor.
For the tensor evolution below, put
\begin{equation}\label{b:sec7_9:8.1}
 T=\Rm_{g_s}-\mathcal M(\Ric_{g_s}),\qquad C_{pij}=\nabla_p(\Ric_{g_s})_{ij},\qquad a(\tau)=\sup_{\Sigma(s)}(|T|+|C|).
\end{equation}
Thus $T$ is the curvature defect and $C=\nabla\Ric_{g_s}$.
Both vanish on the product cylinder.
For a curvature tensor $K$, write
$(\Ric K)_{jl}=\sum_i K_{ijil}$.
For an algebraic curvature tensor $K$, let
\[
 \mathcal S(K)=K-2(K^2+K^\#),
\]
where
\[
 (K^2+K^\#)_{ijkl}=\sum_{p,q}\left(\frac12 K_{ijpq}K_{klpq}+K_{ipkq}K_{jplq}-K_{iplq}K_{jpkq}\right).
\]
Use the metric-compatible time connection $D_\tau$: an orthonormal frame evolves by
$\Ric_{g_s}-\Id/2+
 \bigl((\nabla_{\nabla f}\Ric_g)|_{T\Sigma(s)}\bigr)^\sharp/s$.
For the unforced normalized flow,
\begin{equation}\label{b:sec7_9:8.2}
 D_\tau T=\Delta T-\mathcal F(\Rm_{g_s})+2\sum_p\mathcal M(C_p),
\end{equation}
\[
 \mathcal F(K)=\mathcal S(K)-D\mathcal M_{\Ric_{g_s}}(\Ric\mathcal S(K)),
\]
\begin{equation}\label{b:sec7_9:8.3}
 \begin{split}
 (D_\tau-\Delta)C_{pij}={}&-\frac32 C_{pij}
 +2(\Rm_{g_s})_{ikjl}C_{pkl}
 +2(\Ric_{g_s})_{kl}\nabla_p(\Rm_{g_s})_{ikjl}\\
 &-2(\Rm_{g_s})_{qpir}C_{qrj}
 -2(\Rm_{g_s})_{qpjr}C_{qir}.
 \end{split}
\end{equation}
For \eqref{b:sec7_9:7.4}, both right sides acquire forcing terms.
The extra metric velocity is
$-2(\nabla_{\nabla f}\Ric_g)|_{T\Sigma(s)}/s$.
Variation of curvature uses at most two spatial derivatives of this
velocity; variation of $\nabla\Ric$ uses at most three. Thus, by
\eqref{b:sec7_9:7.6}, for each fixed $m$ the $m$-th spatial
derivatives of both forcing terms are bounded by $C_m s^{-1}$.

Here are the identities behind these formulas. In compatible frames,
$D_\tau\Rm_{g_s}=\Delta\Rm_{g_s}-\mathcal S(\Rm_{g_s})$ and
$D_\tau\Ric_{g_s}=\Delta\Ric_{g_s}-\Ric_{g_s}
+2\Rm_{g_s}(\Ric_{g_s})$. Applying the product rule to
$\mathcal M(\Ric_{g_s})$ yields \eqref{b:sec7_9:8.2}.
For the derivative formula one uses
\[
 \begin{split}
 (\Delta C)_{pij}-\nabla_p\Delta (\Ric_{g_s})_{ij}={}&(\Ric_{g_s})_{pq}C_{qij}
 +2(\Rm_{g_s})_{qpir}C_{qrj}
 +2(\Rm_{g_s})_{qpjr}C_{qir}\\
 &+(\nabla_q(\Rm_{g_s})_{qpir})(\Ric_{g_s})_{rj}
 +(\nabla_q(\Rm_{g_s})_{qpjr})(\Ric_{g_s})_{ir},
 \end{split}
\]
and $\nabla_q(\Rm_{g_s})_{qpir}=C_{ipr}-C_{rpi}$. In fixed coordinates,
\[
 \partial_\tau \Ric_{g_s}=\Delta \Ric_{g_s}
 +2\Rm_{g_s}(\Ric_{g_s})-2\Ric_{g_s}^2,
 \qquad \partial_\tau\Gamma^k_{pi}=-(C_{pik}+C_{ipk}-C_{kpi}).
\]
After differentiation, the connection-variation terms cancel the contracted curvature derivatives. Passing to compatible frames cancels the remaining $\Ric_{g_s}\circ C$ terms; the three covariant slots contribute $-3C/2$. This gives \eqref{b:sec7_9:8.3}. Finally,
\begin{equation}\label{b:sec7_9:8.4}
 \nabla_p\Rm_{g_s}=\nabla_pT+D\mathcal M_{\Ric_{g_s}}(C_p),
\end{equation}
so this is a closed uniformly parabolic system with first-order coupling. The forcing terms are $O(s^{-1})$ by \eqref{b:sec7_9:7.6}.

By Proposition~\ref{cylindrical-compactness}, pointed levels converge to $\mathbb S^2(\sqrt2)\times\mathbb R^2$, up to a finite quotient. The radial identification gives convergence on bounded $\tau$-intervals: $|\partial_\tau g_s|\le C$ controls the metrics, and derivative bounds give $C_\tau^0C_x^m$ compactness. The limit of \eqref{b:sec7_9:7.4} is a complete normalized Ricci flow.

Every limiting time slice is cylindrical. Bounded-curvature uniqueness \cite{Chen-Zhu} identifies its lifted history with
\[
 g_\infty(u)=g_{\mathbb S^2}+e^u g_{\mathbb R^2}.
\]
Composing with $(y,z)\mapsto(y,e^{-u/2}z)$ makes it stationary. In compatible orthonormal frames the resulting diffusion operator is
\begin{equation}\label{b:sec7_9:8.5}
 \Delta_F=\Delta-\frac12z\cdot\nabla,\qquad F=|z|^2/4.
\end{equation}
The dilation respects compatible frames. In particular, $a(\tau)\to0$.

Fix $1/2<\alpha<1$. If $e^{\alpha\tau}a(\tau)$ is unbounded, choose record times $\tau_j\to\infty$, put $a_j=a(\tau_j)>0$, and choose points attaining the supremum. On every fixed backward interval,
\begin{equation}\label{b:sec7_9:8.6}
 \frac{a(\tau_j+u)}{a_j}\le e^{-\alpha u},\qquad \frac{e^{-\tau_j}}{a_j}\longrightarrow0.
\end{equation}
The first inequality holds on each fixed backward interval for large $j$.

\begin{proposition}\label{decay-normalization}
In the above history charts, lifted to cylinder covers and composed with the fixed model dilation, $(T/a_j,C/a_j)$ has a locally smooth subsequential limit on $\mathbb S^2(\sqrt2)\times\mathbb R^2\times(-\infty,0]$. The limit is nonzero at time zero and satisfies
\begin{equation}\label{b:sec7_9:8.7}
 \sup(|T_\infty(u)|+|C_\infty(u)|)\le e^{-\alpha u}.
\end{equation}
It obeys the linearizations of \eqref{b:sec7_9:8.2}--\eqref{b:sec7_9:8.3} and the linearized differential identities of curvature and Ricci.
\end{proposition}
\begin{proof}
Let $\Pi$ be the spectral projector of $\Ric_{g_s}$ onto its two eigenvalues near $1/2$, and put $Q=\Id-\Pi$. Contracting $T$ gives
\[
 \Ric T=2\Ric_{g_s}^2+(1-2R_{g_s})\Ric_{g_s}.
\]
On the two spectral subspaces, this gives
\[
 \Ric_{g_s}=\frac14\bigl[(2R_{g_s}-1)\Id
 +(\Pi-Q)\bigl((2R_{g_s}-1)^2\Id+8\Ric T\bigr)^{1/2}\bigr].
\]
The separated spectral clusters choose the displayed square-root branches. Hence, near the cylinder,
\begin{equation}\label{b:sec7_9:8.8}
 |\Ric_{g_s}-\Pi/2|
 +|\Rm_{g_s}-\mathcal M(\Pi/2)|\le C(|T|+s^{-1}),
\end{equation}
Since $\mathcal F(\mathcal M(\Pi/2))=0$ for every rank-two
plane, its Taylor expansion is a bounded linear coefficient acting on
$T$, plus a bounded coefficient times $R_{g_s}-1$, with quadratic
remainder.

Divide the tensor system by $a_j$. On a fixed interval $[-L,0]$, the remainders are bounded by
\[
 \frac{C(a_je^{\alpha L}+(e^{\tau_j}+1/2)^{-1})^2}{a_j}
 \longrightarrow0.
\]
On the same interval $s=e^{\tau_j+u}+1/2$. For every fixed $m$,
each forcing term divided by $a_j$ has $C^m$ norm at most
$C_{m,L}e^{-\tau_j}/a_j\to0$.
The scalar error $R_{g_s}-1=1/(2s)$ has the same bound.
By \eqref{b:sec7_9:8.4}, the equation for $C$ contains
$\nabla T$ and no independent derivative of $\Rm_{g_s}$.
Thus the system for $(T,C)$ is uniformly parabolic, with
first-order coupling and coefficients bounded in every fixed
spatial order on the lifted cylinder charts.

The record-time bound gives
$\sup(|T|+|C|)/a_j\le e^{\alpha L}$ on $[-L,0]$.
Interior parabolic estimates on slightly larger charts and intervals
give uniform $C^{k,\vartheta}$ bounds for $(T,C)/a_j$ on each
smaller compact cylinder, for every fixed $k$.
The bounds just displayed make every forcing term vanish in the
limit. A diagonal subsequence therefore converges smoothly.
At a chosen maximum point at $u=0$, the sum of the two norms is one;
after centering a lifted chart there, the limit is nonzero.
The supremum bound and the differential identities pass to the limit.
\end{proof}

On the product cylinder, let $\Pi,Q$ project onto its spherical and
flat factors.
Put $\Ric_0=\Pi/2$ and $\Rm_0=\mathcal M(\Ric_0)$.
Indices $a,b$ are spherical, and $\mu,\nu$ are flat.
We suppress the subscript $\infty$ while studying the limiting system.

\begin{lemma}\label{decay-linear-system}
The limiting curvature defect belongs to the parallel thirteen-dimensional space defined by
\begin{equation}\label{b:sec7_9:8.9}
 T_{12,12}=T_{34,34}=0,\qquad \sum_{a,\mu}T_{a\mu,a\mu}=0,\qquad (\Ric T)_{a\mu}=0,
\end{equation}
together with the curvature symmetries and first Bianchi identity. Put $D=\Pi(\Ric T)\Pi-Q(\Ric T)Q$ and $\widehat T=T+D\mathcal M_{\Ric_0}(D)$. Then
\begin{equation}\label{b:sec7_9:8.10}
 \partial_uT=\Delta_FT
 -\mathcal S'_{\Rm_0}(\widehat T)
 +D\mathcal M_{\Ric_0}
    (\Ric\mathcal S'_{\Rm_0}(\widehat T)).
\end{equation}
The negative of the zeroth-order operator in \eqref{b:sec7_9:8.10} is self-adjoint,
with eigenvalues $0,1,2$ of multiplicities $3,4,6$, respectively.
\end{lemma}
\begin{proof}
Work in moving orthonormal frames adapted to the two Ricci
eigenspaces. Write $\delta$ for first variation along the normalized
sequence. The scalar error divided by $a_j$ tends to zero.
Linearizing the contraction identity in the proof of
Proposition~\ref{decay-normalization} at $\Ric_0=\Pi/2$ gives
\[
 D=\delta\Ric,\qquad
 \widehat T=\delta\Rm.
\]
The off-diagonal block of $\delta\Ric$ vanishes in these frames.
The scalar and Ricci-norm identities in \eqref{b:sec7_9:7.5}
therefore give
\[
 \tr D=0,\qquad
 0=\delta|\Ric|^2=2\langle\Ric_0,D\rangle
   =\tr(\Pi D).
\]
To obtain the sectional constraint, use the norm evolution
\begin{equation}\label{b:sec7_9:8.11}
 (\partial_\tau-\Delta)|\Ric_{g_s}|^2=-2|C|^2
 +4\Rm_{g_s}(\Ric_{g_s},\Ric_{g_s})-2|\Ric_{g_s}|^2
 +O\!\left(s^{-1}\sum_{m\le2}
   |\nabla^m\nabla_{\nabla f}\Ric_g|\right),
\end{equation}
Equations \eqref{b:sec7_9:7.6}, \eqref{b:sec7_9:7.7}, and
\eqref{b:sec7_9:8.6} make its first variation equal to zero:
after division by $a_j$, the time and Laplacian terms are
$O(s^{-3/2}/a_j)+O(s^{-2}/a_j)$, the forcing is
$o(s^{-1}/a_j)$, and $|C|^2/a_j=O(a_j)$.
At the cylinder,
\[
 \delta\bigl[\Rm(\Ric,\Ric)\bigr]
 =\tfrac12\widehat T_{12,12}
   +\tfrac12\tr(\Pi D).
\]
Hence $\widehat T_{12,12}=0$.
The $1212$ component of $D\mathcal M_{\Ric_0}(D)$ is
$\tr(\Pi D)=0$, so $T_{12,12}=0$.
Contraction of $T$ now yields
\[
 \tr(\Pi D)=2T_{12,12}+\sum_{a,\mu}T_{a\mu,a\mu}=0,
 \qquad
 \tr D=2T_{12,12}-2T_{34,34}=0.
\]
Thus $T_{34,34}=0$ and the mixed sectional sum vanishes.
The mixed block of the linearized contraction identity is
$(\Ric T)_{a\mu}=0$. This proves \eqref{b:sec7_9:8.9}.
The space of algebraic curvature tensors in dimension four has
dimension $20$; these seven independent conditions leave dimension
$13$. Linearizing \eqref{b:sec7_9:8.2} gives
\eqref{b:sec7_9:8.10}.

We compute the negative zeroth-order operator in \eqref{b:sec7_9:8.10} from the formula
for $\mathcal S$ and the derivative of $\mathcal M$ above.
On the mixed sectional components it is
\[
 T_{a\mu,b\nu}+T_{b\mu,a\nu}
   -\delta_{ab}\sum_cT_{c\mu,c\nu}.
\]
If $T_{a\mu,b\nu}=\delta_{ab}v_{\mu\nu}$ with
$v_{\mu\nu}=v_{\nu\mu}$ and $\sum_\mu v_{\mu\mu}=0$,
the three terms cancel. This gives two zero modes.
The tensor $T_{a\mu,b\nu}=t\varepsilon_{ab}\varepsilon_{\mu\nu}$
gives the third; its $T_{12,34}$ component is fixed by Bianchi.
For a trace-free symmetric spherical matrix times a symmetric flat
matrix, the first two terms equal $2T_{a\mu,b\nu}$ and the trace
term is zero. These form a six-dimensional eigenspace with
eigenvalue $2$.

The remaining four components are $u_{a\mu}=T_{12,a\mu}$.
The mixed Ricci constraints impose
\[
 T_{34,13}=-u_{24},\qquad T_{34,14}=u_{23},\qquad T_{34,23}=u_{14},\qquad T_{34,24}=-u_{13}.
\]
Substitution into the same curvature polynomial gives reaction
$u_{a\mu}$ in each of the four $12,a\mu$ components; the displayed
$34,a\mu$ components have the same reaction.
The three spaces are mutually orthogonal, with dimensions
$3+6+4=13$. This proves the spectrum and self-adjointness.
\end{proof}

\begin{lemma}\label{decay-ancient-modes}
The only possibility for the pair in Proposition~\ref{decay-normalization} is
\begin{equation}\label{b:sec7_9:8.12}
 T=0,\qquad C_{a,b\mu}(u)=e^{u/2}v_\mu\varepsilon_{ab},
\end{equation}
where $v\in\mathbb R^2$ is constant. The last two indices of $C$ are symmetric, and its other components vanish.
\end{lemma}
\begin{proof}
Use the Gaussian measure $e^{-F}dV$. The rough drift Laplacian is
nonpositive. The zeroth-order eigenvalues $-1,-2$ are excluded by the ancient
bound $e^{-\alpha u}$, $\alpha<1$. On the zero-reaction bundle, the
spherical and flat spectral values are
\[
-\frac{\ell(\ell+1)}2,\qquad -\frac j2.
\]
The growth bound leaves $\ell=0$ and $j=0,1$; spatial boundedness
excludes $j=1$. Thus $T$ is parallel and stationary.

Differentiating the contraction identity for $T$ and using differential Bianchi gives
\[
 \Pi C_p+C_p\Pi-C_p=0,\qquad d^\nabla[D\mathcal M_{\Ric_0}(C)]=0,\qquad \sum_p C_{ppj}=0.
\]
In the two-plane splitting, these identities read
\begin{equation}\label{b:sec7_9:8.13}
 C_{\mu,ij}=0,\qquad C_{a,bc}=C_{a,\mu\nu}=0,\qquad \sum_aC_{a,a\mu}=0.
\end{equation}
The linearized derivative commutator is
\begin{equation}\label{b:sec7_9:8.14}
 \begin{split}
 \nabla_pC_{qij}-\nabla_qC_{pij}={}&\widehat T_{pqir}(\Ric_0)_{rj}+\widehat T_{pqjr}(\Ric_0)_{ir}\\
 &+(\Rm_0)_{pqir}D_{rj}+(\Rm_0)_{pqjr}D_{ir}.
 \end{split}
\end{equation}
For $p$ flat and $q$ spherical, its left side is $\partial_pC_q$, while the right side is parallel. A bounded tensor on the entire flat factor cannot have a nonzero constant derivative. Hence this mixed commutator vanishes. If $T_{a\mu,b\nu}=\delta_{ab}v_{\mu\nu}+t\varepsilon_{ab}\varepsilon_{\mu\nu}$, then $\widehat T_{a\mu,b\nu}=-\delta_{ab}v_{\mu\nu}+t\varepsilon_{ab}\varepsilon_{\mu\nu}$. The mixed commutator forces both its symmetric and skew spherical parts to vanish. Thus $T=0$.

For each flat index $\mu$, the matrix $(C_{a,b\mu})$ is trace-free.
Substituting \eqref{b:sec7_9:8.13} and
$\nabla\Rm_{g_s}=D\mathcal M_{\Ric_0}(C)$ into
\eqref{b:sec7_9:8.3} gives
\[
 (\partial_u-\Delta_F)C_{a,b\mu}
 =-\frac12C_{a,b\mu}-C_{b,a\mu}.
\]
The symmetric trace-free part has zeroth-order coefficient $-3/2$
and vanishes by the ancient bound \eqref{b:sec7_9:8.7}. The skew part is
$C_{a,b\mu}=u_\mu\varepsilon_{ab}$, where
$\partial_uu_\mu=\Delta_Fu_\mu+u_\mu/2$. The growth bound leaves
only finitely many spherical and Hermite modes. Spatial boundedness
excludes all nonconstant flat polynomials. The spherical components
of \eqref{b:sec7_9:8.14} now have zero right side and force
$\nabla_{\mathbb S^2}u_\mu=0$. Thus $u_\mu=e^{u/2}v_\mu$, proving
\eqref{b:sec7_9:8.12}.
\end{proof}

\begin{lemma}\label{decay-period}
Let $g_j$ converge smoothly to the product metric on a neighborhood
of $\overline\Omega$, where
$\Omega=\mathbb S^2(\sqrt2)\times B_2^2$.
Suppose $b_j\downarrow0$ and, for some $0<\vartheta<1$,
\[
 \|\Rm_{g_j}-\mathcal M(\Ric_{g_j})\|_{C^{0,\vartheta}(\overline\Omega)}
 +\|\nabla\Ric_{g_j}\|_{C^{0,\vartheta}(\overline\Omega)}
 \le Cb_j,\qquad
 \|R_{g_j}-1\|_{C^0(\overline\Omega)}=o(b_j).
\]
The H\"older norms use fixed product charts.
If $\nabla \Ric_{g_j}/b_j$ converges to $C_{a,b\mu}=v_\mu\varepsilon_{ab}$ with constant $v_\mu$, then $v=0$.
\end{lemma}
\begin{proof}
Let $\Pi_j$ be the high-Ricci projector and $Q_j=\Id-\Pi_j$.
The assumed bounds also give
\[
 dR_{g_j}=\tr_{g_j}(\nabla\Ric_{g_j})=O(b_j),
 \qquad \|R_{g_j}-1\|_{C^{0,\vartheta}}=O(b_j),
\]
Spectral separation, the contraction identity used in
\eqref{b:sec7_9:8.8}, and the assumed H\"older bounds give
\[
 \|\Ric_{g_j}-\Pi_j/2\|_{C^{0,\vartheta}}
 +\|\nabla Q_j\|_{C^{0,\vartheta}}=O(b_j).
\]
Projection from the product normal bundle identifies it with the oriented rank-two bundle $Q_j$. Its induced connection has curvature
\[
 \mathcal F_j^Q=Q_j\Rm_{g_j}Q_j
   +Q_j(\nabla Q_j\wedge\nabla Q_j)Q_j.
\]
The first term is $O(b_j)$ in $C^{0,\vartheta}$: the curvature
defect has that bound, while the low Ricci block is $O(b_j)$.
The second term is $O(b_j^2)$ in the same norm. Hence
$\|\mathcal F_j^Q\|_{C^{0,\vartheta}}=O(b_j)$.
A global reference frame has connection form $\omega_j$, with
$d\omega_j=\mathcal F_j^Q$. Using the fixed product metric, solve
the Neumann gauge equation so that $\omega_j-d\theta_j$ is
coclosed with zero normal component. The harmonic kernel is
$H^1(\Omega;\mathbb R)=0$, so the div--curl estimate implies
\[
 \|\omega_j-d\theta_j\|_{C^{1,\vartheta}}\le C\|\mathcal F_j^Q\|_{C^{0,\vartheta}}=O(b_j).
\]
Rotate by $\theta_j$ and fix the frame at one point. The resulting unit normal sections $v_j^\mu$ satisfy $|\nabla v_j^\mu|=O(b_j)$: the normal part follows from the connection form and the tangential part from $\nabla Q_j$. They converge to constant product normal vectors.

Set $\eta_j^\mu=(v_j^\mu)^\flat$. For high-bundle vectors $X,Y$, $\Ric_{g_j}(v_j^\mu,Y)=0$. Differentiating yields
\[
 (\nabla_X\Ric_{g_j})(v_j^\mu,Y)=-\frac12\langle\nabla_Xv_j^\mu,Y\rangle+O(b_j^2).
\]
Indeed, replacing the high Ricci block by $\Pi_j/2$ costs $O(b_j^2)$; the low Ricci block is $O(b_j)$ and $Q_j\nabla_XY=-(\nabla_XQ_j)Y=O(b_j)$. Antisymmetrization gives
\[
 d\eta_j^\mu(X,Y)=-2\{(\nabla_X\Ric_{g_j})(v_j^\mu,Y)-(\nabla_Y\Ric_{g_j})(v_j^\mu,X)\}+O(b_j^2).
\]
The tangent planes of the fixed sphere converge to the high planes, and $d\eta_j^\mu=O(b_j)$ on all planes. Stokes' theorem on that actual closed sphere therefore gives
\[
 0=\frac1{b_j}\int_{\mathbb S^2\times\{0\}}d\eta_j^\mu\longrightarrow-4v_\mu|\mathbb S^2(\sqrt2)|.
\]
Thus every $v_\mu$ vanishes.
\end{proof}

\begin{theorem}
For every $\alpha<1$,
\begin{equation}\label{b:sec7_9:8.15}
 \sup_{\Sigma(s)}\bigl(
  |\Rm_{g_s}-\mathcal M(\Ric_{g_s})|
  +|\nabla\Ric_{g_s}|\bigr)=O(s^{-\alpha}).
\end{equation}
\end{theorem}
\begin{proof}
For $1/2<\alpha<1$, failure gives the nonzero pair of
Proposition~\ref{decay-normalization}.
Lemma~\ref{decay-ancient-modes} reduces it to
\eqref{b:sec7_9:8.12}. On a fixed lifted sphere neighborhood at
time zero, apply Lemma~\ref{decay-period} with $b_j=a_j$.
Take the neighborhood slightly larger than $\overline\Omega$.
The interior $C^{k,\vartheta}$ estimates for $(T,C)/a_j$ in the
proof of Proposition~\ref{decay-normalization} give precisely the
scaled H\"older bounds in that lemma. Also
$R_{g_j}-1=O(e^{-\tau_j})=o(a_j)$ by
\eqref{b:sec7_9:8.6}. Thus the remaining mode vanishes,
contradicting the basepoint normalization. Smaller exponents
follow immediately.
\end{proof}

\subsection{Total geodesy of the focal surface}
We insert the decay estimate into \eqref{b:sec7_9:7.10}.
The normal two-plane has curvature of order $s^{-1}$; it must be
included in the comparison.

On a large level, choose an orthonormal frame in which
$\{e_1,e_2\}$ spans the high-Ricci plane and
$\{e_3,e_4\}$ its orthogonal complement.
Let $\Rm_{\mathrm{ref}}$ be the algebraic product curvature tensor with
\[
 (\Rm_{\mathrm{ref}})_{1212}=\frac12,\qquad
 (\Rm_{\mathrm{ref}})_{3434}=\frac1{4s},
\]
and all mixed components zero. Its Ricci tensor
$\Ric_{\mathrm{ref}}$ has eigenvalues
$\frac12,\frac12,\frac1{4s},\frac1{4s}$ in this frame.
The preceding decay estimate gives, for $1/2<\alpha<1$,
\begin{equation}\label{b:sec7_9:9.1}
 |\Rm_{g_s}-\Rm_{\mathrm{ref}}|
 +|\Ric_{g_s}-\Ric_{\mathrm{ref}}|
 +|\nabla\Ric_{g_s}|=O(s^{-\alpha}).
\end{equation}
Here $\Ric(\Rm_{g_s}-\Rm_{\mathrm{ref}})
=\Ric_{g_s}-\Ric_{\mathrm{ref}}$.
Direct calculation gives
\[
 R(\Rm_{\mathrm{ref}})=R_{g_s},\qquad
 |\Ric_{\mathrm{ref}}|^2=\frac12+\frac1{8s^2},\qquad
 |\Rm_{\mathrm{ref}}|^2=1+\frac1{4s^2}.
\]
Thus $\tr(\Ric_{g_s}-\Ric_{\mathrm{ref}})=0$.
Using \eqref{b:sec7_9:7.5}, we obtain
\begin{equation}\label{b:sec7_9:9.2}
 2\left(\frac12-\frac1{4s}\right)
 \sum_{a=1}^2\bigl((\Ric_{g_s})_{aa}-\tfrac12\bigr)
 +|\Ric_{g_s}-\Ric_{\mathrm{ref}}|^2
 =\frac{\bigl|(\nabla_{\nabla f}\Ric_g)|_{T\Sigma(s)}\bigr|^2}{s^2}.
\end{equation}
In particular,
\[
 \sum_{a=1}^2\bigl((\Ric_{g_s})_{aa}-\tfrac12\bigr)
 =O\bigl(|\Ric_{g_s}-\Ric_{\mathrm{ref}}|^2+s^{-2}\bigr).
\]

The norm evolution \eqref{b:sec7_9:8.11} also gives
\begin{equation}\label{b:sec7_9:9.3}
 \begin{split}
 \Rm_{g_s}(\Ric_{g_s},\Ric_{g_s})
   -\frac14-\frac12|\nabla\Ric_{g_s}|^2
 &=\frac14(\partial_\tau-\Delta)|\Ric_{g_s}|^2
   +\frac12(|\Ric_{g_s}|^2-1/2)+o(s^{-1})\\
 &=o(s^{-1})
 \end{split}
\end{equation}
uniformly over the entire level. The two explicit terms are
$O(s^{-3/2})+O(s^{-2})$ and $O(s^{-2})$ by
\eqref{b:sec7_9:7.6} and \eqref{b:sec7_9:7.7}.

Expand around $(\Rm_{\mathrm{ref}},\Ric_{\mathrm{ref}})$.
Directly,
\[
 \Rm_{\mathrm{ref}}(\Ric_{\mathrm{ref}},\Ric_{\mathrm{ref}})
   =\frac14+\frac1{32s^3}.
\]
By \eqref{b:sec7_9:9.2},
\[
 \bigl|\Rm_{\mathrm{ref}}
   (\Ric_{\mathrm{ref}},\Ric_{g_s}-\Ric_{\mathrm{ref}})\bigr|
 \le C\bigl(
  |\Rm_{g_s}-\Rm_{\mathrm{ref}}|^2+s^{-2}\bigr).
\]
Together with \eqref{b:sec7_9:9.3}, this yields
\[
 \bigl|(\Rm_{g_s}-\Rm_{\mathrm{ref}})
      (\Ric_{\mathrm{ref}},\Ric_{\mathrm{ref}})\bigr|
 \le C\bigl(
  |\nabla\Ric_{g_s}|^2
  +|\Rm_{g_s}-\Rm_{\mathrm{ref}}|^2\bigr)+o(s^{-1}).
\]
The coefficient of
$(\Rm_{g_s}-\Rm_{\mathrm{ref}})_{1212}$ in this
contraction is $1/2$; all others are $O(s^{-1})$. Hence
\[
 |(\Rm_{g_s}-\Rm_{\mathrm{ref}})_{1212}|
 \le C\bigl(
  |\nabla\Ric_{g_s}|^2
  +|\Rm_{g_s}-\Rm_{\mathrm{ref}}|^2
  +s^{-1}|\Rm_{g_s}-\Rm_{\mathrm{ref}}|\bigr)+o(s^{-1}).
\]
In $2\langle\Rm_{\mathrm{ref}},
\Rm_{g_s}-\Rm_{\mathrm{ref}}\rangle$, the only
coefficient of order one is this high-plane component. Therefore
\begin{equation}\label{b:sec7_9:9.4}
 |\Rm_{g_s}|^2-|\Rm_{\mathrm{ref}}|^2
 \le C\bigl(
  |\nabla\Ric_{g_s}|^2
  +|\Rm_{g_s}-\Rm_{\mathrm{ref}}|^2
  +s^{-1}|\Rm_{g_s}-\Rm_{\mathrm{ref}}|\bigr)+o(s^{-1}).
\end{equation}
Insert this into \eqref{b:sec7_9:7.10}. Its radial-derivative
term is nonpositive. Since $\vol(\Sigma(s))=16\pi|M_-|s$,
\[
 0\le8\pi\int_{M_-}|\II|^2\,dA
 \le Cs^{1-2\alpha}+Cs^{-\alpha}+o(1)\longrightarrow0.
\]
Therefore $\II=0$, and we have proved the following.

\begin{proposition}\label{focal-totally-geodesic}
The focal surface $M_-$ is totally geodesic. On the universal cover, its area is $8\pi$.
\end{proposition}

\begin{remark}
Uniformity of the little-o term in \eqref{b:sec7_9:9.3} is necessary: the volume of a level grows linearly in $s$. The estimate $O(s^{-1})$ alone would leave a nonzero error after integration.
\end{remark}

\section{Rigidity and classification}
\label{n:sec:rigidity}

The preceding sections give $|\Rm|<\infty$ and $|M_-|=8\pi$ on the
universal cover. We now prove $\nabla\Ric=0$ by an integral identity on
$\Sigma(s)$. The splitting then follows from the parallel Ricci tensor.

We work on the universal cover.
The regular levels $\Sigma(s)=f^{-1}(s)$ are compact for $s>0$.
By \eqref{focal-volume} and Proposition~\ref{focal-totally-geodesic}, we have
\begin{equation}\label{n:eq:level-volume}
 |M_-|=8\pi,\qquad
 \vol(\Sigma(s))=16\pi |M_-|s=128\pi^2s.
\end{equation}
Each level is an $S^2$-bundle over $M_-\cong S^2$. In particular,
\begin{equation}\label{n:eq:euler}
 \chi(\Sigma(s))=4.
\end{equation}

We begin with the Ricci eigenvalues. Choose an
orthonormal basis $\{e_1,e_2,e_3,e_4,e_5\}$ such that
\[
 e_1=\frac{\nabla f}{\sqrt f},\qquad
 \Ric^\sharp e_i=\lambda_i e_i,\qquad \lambda_1=0.
\]
In this section all indices $i,j$ run from $2$ to $5$. Let $p_j$
denote the elementary symmetric functions of
$\lambda_2,\lambda_3,\lambda_4,\lambda_5$, and put
\[
 \sigma_j=\sum_i\lambda_i^j.
\]
These functions are smooth on $M\setminus M_-$. Indeed, they are invariant
polynomials of the smooth tangential Ricci endomorphism. The eigenvalues
and the eigenbasis need not be differentiated. From
\eqref{basicR1} we have
\begin{equation}\label{n:eq:power-sums}
 p_1=\sigma_1=1,\qquad
 \sigma_2=\frac12,\qquad p_2=\frac14.
\end{equation}

\begin{lemma}\label{n:lem:defect}
On $M\setminus M_-$ define
\begin{equation}\label{n:eq:Phi}
 \Phi=\sum_i\lambda_i^2\left(\frac12-\lambda_i\right)^2.
\end{equation}
Then
\begin{equation}\label{n:eq:newton}
 \sigma_3=\frac14+3p_3,\qquad
 \sigma_4=\frac18+4p_3-4p_4,
\end{equation}
and
\begin{equation}\label{n:eq:defect-identity}
 \Phi=p_3-4p_4\geq0.
\end{equation}
Moreover, $\Phi=0$ if and only if every tangential Ricci eigenvalue is
$0$ or $1/2$.
\end{lemma}

\begin{proof}
Newton's identities give
\[
 \sigma_3=p_1\sigma_2-p_2\sigma_1+3p_3,
 \qquad
 \sigma_4=p_1\sigma_3-p_2\sigma_2+p_3\sigma_1-4p_4.
\]
Substituting \eqref{n:eq:power-sums} proves \eqref{n:eq:newton}. Therefore
\begin{align*}
 \Phi
 &=\frac14\sigma_2-\sigma_3+\sigma_4\\
 &=\frac18-\left(\frac14+3p_3\right)
   +\left(\frac18+4p_3-4p_4\right)
 =p_3-4p_4.
\end{align*}
The remaining assertions follow from the sum of squares in
\eqref{n:eq:Phi}. Notice also that $\Phi$ is the squared norm of the
ambient polynomial endomorphism $\Ric^\sharp(\Ric^\sharp-\tfrac12\operatorname{Id})$.
It consequently has a smooth extension across the focal surface.
\end{proof}

The next formulas express the intrinsic curvature of $\Sigma(s)$ in
terms of ambient Ricci curvature and its radial derivative.
We identify symmetric two-tensors with self-adjoint endomorphisms when
they occur in matrix expressions. Differentiating $\Ric^\sharp(\nabla f)=0$ along $\nabla f$ and
using $\nabla_{\nabla f}\nabla f=\nabla f/2$ gives
\begin{equation}\label{n:eq:D-tangential}
 (\nabla_{\nabla f}\Ric)(\nabla f,\cdot)=0.
\end{equation}
Differentiating the constant trace and squared norm of the
Ricci tensor gives
\begin{equation}\label{n:eq:D-traces}
 \operatorname{tr}(\nabla_{\nabla f}\Ric)=0,\qquad
 \langle\Ric,\nabla_{\nabla f}\Ric\rangle=0.
\end{equation}

\begin{lemma}\label{n:lem:level-formulas}
Let $\Ric^{\Sigma}$, $R^{\Sigma}$, and $W^{\Sigma}$ be
the intrinsic Ricci tensor, scalar curvature, and Weyl tensor of
$\Sigma(s)$. At a point of $\Sigma(s)$ write $K_{ij}=R_{ijij}$.
Then
\begin{align}
 \Ric^{\Sigma}
 &=\Ric|_{T\Sigma}
   -\frac{(\nabla_{\nabla f}\Ric)|_{T\Sigma}}s
   +\frac{\tfrac12g|_{T\Sigma}-\Ric|_{T\Sigma}}{2s},
       \label{n:eq:intrinsic-Ricci}\\
 R^{\Sigma}&=1+\frac1{2s},\qquad
 |\Ric^{\Sigma}|^2
 =\frac12+\frac{|\nabla_{\nabla f}\Ric|^2+\tfrac18}{s^2},
       \label{n:eq:intrinsic-norm}\\
 K_{1i}&=\frac{(\nabla_{\nabla f}\Ric)(e_i,e_i)
                  +\lambda_i(\tfrac12-\lambda_i)}s,
       \label{n:eq:radial-K}\\
 K_{ij}
 &=\frac12(\lambda_i+\lambda_j)
   -\frac12(K_{1i}+K_{1j})
   -\frac16\left(1+\frac1{2s}\right)+W^{\Sigma}_{ijij}
   \notag\\[-2pt]
 &\hspace{10mm}
   +\frac1{2s}
      \left[\left(\frac12-\lambda_i\right)
            +\left(\frac12-\lambda_j\right)\right]
       (\lambda_i+\lambda_j),\qquad i\ne j.
       \label{n:eq:tangential-K}
\end{align}
\end{lemma}

\begin{proof}
The shape operator of $\Sigma(s)$, with respect to $e_1$, is
\[
 \frac{\tfrac12\operatorname{Id}-\Ric^\sharp}{\sqrt s}.
\]
Its trace is $1/\sqrt s$. The identities
$\nabla_{e_1}e_1=0$, $\nabla f=\sqrt s\,e_1$, and $e_1(s)=\sqrt s$ imply
\[
 \nabla_{e_1}
 \left(\frac{\tfrac12\operatorname{Id}-\Ric^\sharp}{\sqrt s}\right)
 =-\frac{(\nabla_{\nabla f}\Ric)^{\sharp}}s
   -\frac{\tfrac12\operatorname{Id}-\Ric^\sharp}{2s}.
\]
The radial Riccati equation gives, on $T\Sigma(s)$,
\begin{align*}
 R(\,\cdot\,,e_1)e_1
 &= -\nabla_{e_1}
    \left(\frac{\tfrac12\operatorname{Id}-\Ric^\sharp}{\sqrt s}\right)
    -\left(\frac{\tfrac12\operatorname{Id}-\Ric^\sharp}{\sqrt s}\right)^2\\
 &=\frac1s\left((\nabla_{\nabla f}\Ric)^{\sharp}
                   +\frac12\Ric^\sharp-(\Ric^\sharp)^2\right).
\end{align*}
This proves \eqref{n:eq:radial-K}.

The contracted Gauss equation now reads
\[
 \Ric^{\Sigma}
 =\Ric|_{T\Sigma}-
  \bigl(R(\,\cdot\,,e_1)e_1\bigr)^{\flat}
  +\frac{\tfrac14g|_{T\Sigma}-(\Ric^\sharp)^2{}^{\flat}}s.
\]
Substitution proves \eqref{n:eq:intrinsic-Ricci}. Its trace gives the
scalar formula in \eqref{n:eq:intrinsic-norm}. For the norm formula,
\eqref{basicR1} and \eqref{n:eq:D-traces} show that
\[
 \left|\frac12\operatorname{Id}-\Ric^\sharp\right|_{T\Sigma}^2=\frac12,
 \qquad
 \left\langle \Ric^\sharp,\frac12\operatorname{Id}-\Ric^\sharp\right\rangle=0,
 \qquad
 \left\langle\nabla_{\nabla f}\Ric,\frac12\operatorname{Id}-\Ric^\sharp\right\rangle=0.
\]
All cross terms therefore vanish when
\eqref{n:eq:intrinsic-Ricci} is squared.

Finally, the four-dimensional Weyl decomposition and the Gauss equation
give
\[
 K_{ij}
 =W^{\Sigma}_{ijij}
   +\frac12\bigl(\Ric^{\Sigma}_{ii}
                 +\Ric^{\Sigma}_{jj}\bigr)
   -\frac16R^{\Sigma}
   -\frac{(\tfrac12-\lambda_i)(\tfrac12-\lambda_j)}s.
\]
Use \eqref{n:eq:intrinsic-Ricci} and
\eqref{n:eq:radial-K} to cancel the radial-derivative terms.
The terms with the factor
$s^{-1}$ combine as
\[
 \frac1{2s}(1-\lambda_i-\lambda_j)(\lambda_i+\lambda_j).
\]
This is exactly the last term in \eqref{n:eq:tangential-K}.
\end{proof}

The Weyl term in the level curvature has the following sharp bound.

\begin{lemma}\label{n:lem:Weyl}
With the full tensor norm of the Weyl tensor,
\begin{equation}\label{n:eq:Weyl-bound}
 (2\sum_{i\ne j}W^{\Sigma}_{ijij}\lambda_i\lambda_j)^2\leq\left(\frac1{12}-4\Phi\right)|W^{\Sigma}|^2,
 \qquad 0\leq\Phi\leq\frac1{48}.
\end{equation}
\end{lemma}

\begin{proof}
In the orthonormal basis $e_2,e_3,e_4,e_5$, the trace-free property of a
four-dimensional Weyl tensor gives
\[
\begin{aligned}
 W_{2323}^{\Sigma}=W_{4545}^{\Sigma}=a,&\qquad W_{2424}^{\Sigma}=W_{3535}^{\Sigma}=b,\\
 W_{2525}^{\Sigma}=W_{3434}^{\Sigma}=c,&\qquad a+b+c=0.
\end{aligned}
\]
The diagonal components contribute $8(a^2+b^2+c^2)$ to the full norm.
The remaining components have nonnegative squared contribution, so
\begin{equation}\label{n:eq:Weyl-norm}
 |W^{\Sigma}|^2\geq8(a^2+b^2+c^2).
\end{equation}
Define
\[
 Q_1=\lambda_2\lambda_3+\lambda_4\lambda_5,
 \quad Q_2=\lambda_2\lambda_4+\lambda_3\lambda_5,
 \quad Q_3=\lambda_2\lambda_5+\lambda_3\lambda_4.
\]
Then $Q_1+Q_2+Q_3=p_2=1/4$ and
\begin{align*}
 2\sum_{i\ne j}W^{\Sigma}_{ijij}\lambda_i\lambda_j
 &=4(aQ_1+bQ_2+cQ_3)\\
 &=4\left[a\left(Q_1-\frac1{12}\right)
           +b\left(Q_2-\frac1{12}\right)
           +c\left(Q_3-\frac1{12}\right)\right].
\end{align*}
Cauchy--Schwarz and \eqref{n:eq:Weyl-norm} imply
\begin{equation}\label{n:eq:Weyl-CS}
 (2\sum_{i\ne j}W^{\Sigma}_{ijij}\lambda_i\lambda_j)^2\leq2|W^{\Sigma}|^2
       \sum_{\alpha=1}^3\left(Q_\alpha-\frac1{12}\right)^2.
\end{equation}
Expanding the products gives
\[
 \sum_{\alpha=1}^3Q_\alpha^2=p_2^2-2p_1p_3+8p_4.
\]
Consequently,
\begin{align*}
 2\sum_{\alpha=1}^3\left(Q_\alpha-\frac1{12}\right)^2
 &=2\sum_{\alpha=1}^3Q_\alpha^2-\frac1{24}\\
 &=2\left(\frac1{16}-2p_3+8p_4\right)-\frac1{24}\\
 &=\frac1{12}-4p_3+16p_4
 =\frac1{12}-4\Phi.
\end{align*}
This expression is nonnegative because it is a sum of squares. Together
with Lemma~\ref{n:lem:defect}, this proves both assertions.
\end{proof}

We next use \eqref{n:eq:tangential-K} in the Ricci Bochner identity.

\begin{proposition}\label{n:prop:pointwise}
On every regular level $\Sigma(s)$,
\begin{equation}\label{n:eq:pointwise}
 |\nabla\Ric|^2=2\sum_{i\ne j}W^{\Sigma}_{ijij}\lambda_i\lambda_j-\frac16-6p_3+\frac2s\nabla f\cdot\nabla p_3
     +\frac1s\left(5p_3-8p_4-\frac1{12}\right).
\end{equation}
\end{proposition}

\begin{proof}
Since $|\Ric|^2=1/2$ is constant, the Ricci Bochner formula \eqref{ric-bochner} gives
\[
 |\nabla\Ric|^2=2\sum_{i\ne j}K_{ij}\lambda_i\lambda_j-\frac12.
\]
The radial Ricci eigenvalue is zero. Substitute
\eqref{n:eq:tangential-K} and calculate its four non-Weyl terms.
First, \eqref{n:eq:power-sums} and \eqref{n:eq:newton} give
\[
 \sum_{i\ne j}(\lambda_i+\lambda_j)\lambda_i\lambda_j-\frac12
 =2(\sigma_1\sigma_2-\sigma_3)-\frac12=-6p_3.
\]

Next, using \eqref{n:eq:radial-K} and $\sum_{j\ne i}\lambda_j=1-\lambda_i$,
\begin{align*}
 -\sum_{i\ne j}(K_{1i}+K_{1j})\lambda_i\lambda_j
 &=-2\sum_iK_{1i}\lambda_i(1-\lambda_i)\\
 &=-\frac2s\sum_i
   (\nabla_{\nabla f}\Ric)(e_i,e_i)\lambda_i(1-\lambda_i)
   -\frac2s\sum_i\lambda_i^2\left(\frac12-\lambda_i\right)(1-\lambda_i).
\end{align*}
To calculate the first sum, differentiate traces of powers of $\Ric^\sharp$.
Equation \eqref{n:eq:D-tangential} gives
\[
 \nabla f\cdot\nabla\sigma_2
 =2\sum_i\lambda_i(\nabla_{\nabla f}\Ric)(e_i,e_i)=0,
 \qquad
 \nabla f\cdot\nabla\sigma_3
 =3\sum_i\lambda_i^2(\nabla_{\nabla f}\Ric)(e_i,e_i)
 =3\nabla f\cdot\nabla p_3.
\]
Thus
\[
 \sum_i(\nabla_{\nabla f}\Ric)(e_i,e_i)
       \lambda_i(1-\lambda_i)
 =-\nabla f\cdot\nabla p_3.
\]
This calculation uses smooth invariant functions and is valid at repeated
Ricci eigenvalues as well. For the second sum,
\begin{align*}
 \sum_i\lambda_i^2\left(\frac12-\lambda_i\right)(1-\lambda_i)
 &=\frac12\sigma_2-\frac32\sigma_3+\sigma_4\\
 &=-\frac12p_3-4p_4.
\end{align*}
Therefore
\[
 -\sum_{i\ne j}(K_{1i}+K_{1j})\lambda_i\lambda_j
 =\frac2s\nabla f\cdot\nabla p_3+\frac1s(p_3+8p_4).
\]

Since $\sum_{i\ne j}\lambda_i\lambda_j=2p_2=1/2$,
\[
 -\frac13\left(1+\frac1{2s}\right)
 \sum_{i\ne j}\lambda_i\lambda_j
 =-\frac16\left(1+\frac1{2s}\right).
\]
Finally,
\[
 \sum_{i\ne j}(\lambda_i+\lambda_j)\lambda_i\lambda_j
       =2(\sigma_1\sigma_2-\sigma_3),
\]
and
\[
 \sum_{i\ne j}(\lambda_i+\lambda_j)^2\lambda_i\lambda_j
       =2(\sigma_1\sigma_3+\sigma_2^2-2\sigma_4).
\]
It follows that
\begin{align*}
 \frac1s\sum_{i\ne j}
  (1-\lambda_i-\lambda_j)(\lambda_i+\lambda_j)\lambda_i\lambda_j
 &=\frac2s\left[(\sigma_1\sigma_2-\sigma_3)
               -(\sigma_1\sigma_3+\sigma_2^2-2\sigma_4)\right]\\
 &=\frac1s(4p_3-16p_4).
\end{align*}
Adding these four expressions and the Weyl term proves
\eqref{n:eq:pointwise}.
\end{proof}

The area $|M_-|=8\pi$ now enters through Gauss--Bonnet on $\Sigma(s)$.
All tensor norms are full norms.

\begin{lemma}\label{n:lem:Weyl-integral}
For every $s>0$,
\begin{equation}\label{n:eq:Weyl-integral}
 \int_{\Sigma(s)}|W^{\Sigma}|^2\,d\sigma
 =\int_{\Sigma(s)}
    \left[\frac13\left(1+\frac1{2s}\right)^2
          +\frac{2|\nabla_{\nabla f}\Ric|^2}{s^2}\right]d\sigma.
\end{equation}
\end{lemma}

\begin{proof}
The four-dimensional Gauss--Bonnet--Chern formula, in the convention of
Lemma~\ref{level-curvature}, is
\[
 32\pi^2\chi(\Sigma(s))
 =\int_{\Sigma(s)}
     \left(|\Rm^{\Sigma}|^2
           -4|\Ric^{\Sigma}|^2+(R^{\Sigma})^2\right)d\sigma.
\]
The orthogonal curvature decomposition in dimension four gives
\[
 |\Rm^{\Sigma}|^2
 =|W^{\Sigma}|^2+2|\Ric^{\Sigma}|^2
                  -\frac13(R^{\Sigma})^2.
\]
By \eqref{n:eq:euler},
\begin{align*}
 \int_{\Sigma(s)}|W^{\Sigma}|^2\,d\sigma
 &=128\pi^2+
   \int_{\Sigma(s)}
      \left(2|\Ric^{\Sigma}|^2
                  -\frac23(R^{\Sigma})^2\right)d\sigma\\
 &=128\pi^2+
   \int_{\Sigma(s)}
      \left(\frac13-\frac2{3s}+\frac1{12s^2}
                         +\frac{2|\nabla_{\nabla f}\Ric|^2}{s^2}\right)d\sigma.
\end{align*}
The volume formula \eqref{n:eq:level-volume} implies
$128\pi^2=\vol(\Sigma(s))/s$. Combining this contribution
with the term $-2/(3s)$ leaves $1/(3s)$. The scalar terms are then
\[
 \frac13+\frac1{3s}+\frac1{12s^2}
 =\frac13\left(1+\frac1{2s}\right)^2,
\]
which proves the identity.
\end{proof}

\begin{proposition}\label{n:prop:integrated}
For every $s>0$,
\begin{align}\label{n:eq:integrated}
 \frac{2s}{2s+1}\int_{\Sigma(s)}|\nabla\Ric|^2\,d\sigma
 &\leq\left(-6+\frac3s\right)
                 \int_{\Sigma(s)}p_3\,d\sigma
      +\frac2s\int_{\Sigma(s)}\nabla f\cdot\nabla p_3\,d\sigma\\
 &\quad-4\int_{\Sigma(s)}\Phi\,d\sigma.
\end{align}
\end{proposition}

\begin{proof}
Set $a=1+1/(2s)$. Lemma~\ref{n:lem:Weyl} and Young's inequality give
\begin{equation}\label{n:eq:young}
 2\sum_{i\ne j}W^{\Sigma}_{ijij}\lambda_i\lambda_j\leq\sqrt{\frac1{12}-4\Phi}\,|W^{\Sigma}|
 \leq a\left(\frac1{12}-4\Phi\right)
                   +\frac{|W^{\Sigma}|^2}{4a}.
\end{equation}
Integrating and applying Lemma~\ref{n:lem:Weyl-integral}, we obtain
\begin{align*}
 \int_{\Sigma(s)}(2\sum_{i\ne j}W^{\Sigma}_{ijij}\lambda_i\lambda_j)\,d\sigma
 &\leq \frac a{12}\vol(\Sigma(s))
       -4a\int_{\Sigma(s)}\Phi\,d\sigma\\
 &\qquad
       +\frac1{4a}\int_{\Sigma(s)}
                      \left(\frac{a^2}3+\frac{2|\nabla_{\nabla f}\Ric|^2}{s^2}\right)d\sigma\\
 &=\frac a6\vol(\Sigma(s))
       -4a\int_{\Sigma(s)}\Phi\,d\sigma
       +\frac1{2as^2}\int_{\Sigma(s)}|\nabla_{\nabla f}\Ric|^2\,d\sigma.
\end{align*}
Pointwise, Cauchy--Schwarz and $|\nabla f|^2=s$ imply
\[
 |\nabla_{\nabla f}\Ric|^2
        \leq s|\nabla\Ric|^2.
\]
Since $2as=2s+1$, this proves
\begin{equation}\label{n:eq:Weyl-contraction-integral}
 \int_{\Sigma(s)}(2\sum_{i\ne j}W^{\Sigma}_{ijij}\lambda_i\lambda_j)\,d\sigma
 \leq\frac a6\vol(\Sigma(s))
       -4a\int_{\Sigma(s)}\Phi\,d\sigma
       +\frac1{2s+1}\int_{\Sigma(s)}|\nabla\Ric|^2\,d\sigma.
\end{equation}

By \eqref{n:eq:defect-identity}, $5p_3-8p_4=3p_3+2\Phi$.
Integrating \eqref{n:eq:pointwise} and using
\eqref{n:eq:Weyl-contraction-integral} therefore gives
\begin{align*}
 \int_{\Sigma(s)}|\nabla\Ric|^2\,d\sigma
 &\leq\left(\frac a6-\frac16-\frac1{12s}\right)
                      \vol(\Sigma(s))\\
 &\quad+\left(-6+\frac3s\right)\int_{\Sigma(s)}p_3\,d\sigma
        +\frac2s\int_{\Sigma(s)}\nabla f\cdot\nabla p_3\,d\sigma\\
 &\quad+\left(-4a+\frac2s\right)\int_{\Sigma(s)}\Phi\,d\sigma
        +\frac1{2s+1}\int_{\Sigma(s)}|\nabla\Ric|^2\,d\sigma.
\end{align*}
The volume coefficient is zero, and $-4a+2/s=-4$. Moving the last term
to the left proves \eqref{n:eq:integrated}.
\end{proof}

\begin{proposition}\label{n:prop:monotonicity}
Define
\begin{equation}\label{n:eq:G}
 P(s)=\int_{\Sigma(s)}p_3\,d\sigma,
 \qquad G(s)=e^{-3s}s^{1/2}P(s),\qquad s>0.
\end{equation}
Then
\begin{equation}\label{n:eq:G-prime}
 G'(s)\geq e^{-3s}s^{3/2}
   \left[\frac1{2s+1}\int_{\Sigma(s)}|\nabla\Ric|^2\,d\sigma
         +\frac2s\int_{\Sigma(s)}\Phi\,d\sigma\right]\geq0.
\end{equation}
Moreover, $G\equiv0$, and
\begin{equation}\label{n:eq:vanishing}
 \nabla\Ric=0,\qquad \Phi=0
 \quad\hbox{on }M.
\end{equation}
\end{proposition}

\begin{proof}
On a positive $f$-slab, the flow of $\nabla f/f$ identifies nearby
regular levels. Since $(\nabla f/f)(f)=1$, its parameter is the level
value $s$. If $g_s$ denotes the pulled-back induced metric, then
\[
 \frac{\partial}{\partial s}g_s
 =\frac2s\nabla^2f|_{T\Sigma(s)}
 =\frac1s\left(g_s-2\Ric|_{T\Sigma(s)}\right).
\]
The possible derivatives of $f^{-1}$ in the Lie derivative vanish on
tangential arguments. Taking half the trace in dimension four gives
\begin{equation}\label{n:eq:volume-variation}
 \frac{\partial}{\partial s}d\sigma_s
 =\frac1{2s}(4-2)d\sigma_s=\frac1s d\sigma_s.
\end{equation}
It follows that
\begin{equation}\label{n:eq:P-prime}
 P'(s)=\frac1s\int_{\Sigma(s)}\nabla f\cdot\nabla p_3\,d\sigma+\frac1sP(s).
\end{equation}
Substituting \eqref{n:eq:P-prime} in
\eqref{n:eq:integrated} yields
\[
 \frac{2s}{2s+1}\int_{\Sigma(s)}|\nabla\Ric|^2\,d\sigma
 \leq2P'(s)+\left(-6+\frac1s\right)P(s)
                  -4\int_{\Sigma(s)}\Phi\,d\sigma.
\]
Equivalently,
\[
 P'(s)+\left(\frac1{2s}-3\right)P(s)
 \geq\frac{s}{2s+1}\int_{\Sigma(s)}|\nabla\Ric|^2\,d\sigma
                  +2\int_{\Sigma(s)}\Phi\,d\sigma.
\]
Multiplication by $e^{-3s}s^{1/2}$ proves
\eqref{n:eq:G-prime}.

It remains to evaluate $G$ at the endpoints.
Since $\sum_i\lambda_i^2=1/2$, we have $|p_3|\le C$. By
\eqref{n:eq:level-volume},
\[
 |P(s)|\leq C\vol(\Sigma(s))\leq Cs.
\]
Therefore
\[
 |G(s)|\leq Ce^{-3s}s^{3/2},\qquad
 \lim_{s\downarrow0}G(s)=0,
 \qquad \lim_{s\to\infty}G(s)=0.
\]
A nondecreasing function with these two endpoint limits is identically
zero. Thus the nonnegative right-hand side of
\eqref{n:eq:G-prime} vanishes for every $s>0$. In particular,
\[
 \int_{\Sigma(s)}|\nabla\Ric|^2\,d\sigma=0,
 \qquad
 \int_{\Sigma(s)}\Phi\,d\sigma=0.
\]
Both integrands are continuous and nonnegative, so they vanish pointwise
on every regular level. The set $M\setminus M_-$ is dense in $M$.
Continuity of $\nabla\Ric$ and the polynomial expression
for $\Phi$ in Lemma~\ref{n:lem:defect} extend the vanishing across $M_-$.
\end{proof}

\begin{proof}[Proof of Theorem~\ref{maintheorem}]
By Proposition~\ref{n:prop:monotonicity}, the Ricci tensor is parallel.
Also $\Phi=0$, so every Ricci eigenvalue is either zero or $1/2$.
Since $R=1$, exactly two eigenvalues are $1/2$ and the other three are
zero. The corresponding eigendistributions are orthogonal and parallel.
The de Rham decomposition theorem on the complete simply connected
universal cover gives
\[
 (M,g)\cong (N^2,g_N)\times(P^3,g_P),\qquad
 \Ric_{g_N}=\frac12g_N,
 \qquad \Ric_{g_P}=0.
\]
The factors are complete and simply connected. In dimension three the
Riemann tensor is determined by the Ricci tensor, so $(P^3,g_P)$ is flat
and isometric to $\mathbb R^3$. The surface $(N^2,g_N)$ has constant
Gaussian curvature $1/2$, and is therefore the round sphere
$S^2(\sqrt2)$. Consequently,
\[
 (M,g)\cong S^2(\sqrt2)\times\mathbb R^3
\]
on the universal cover.

The identity $\Ric(\nabla f,\cdot)=0$ shows that $f$ is
independent of the spherical variable. Its Euclidean Hessian is
$g_{\mathbb R^3}/2$, so
\[
 f(z)=\frac{|z|^2}{4}+\langle b,z\rangle+c.
\]
After translating the Euclidean origin, this becomes
$f(z)=|z|^2/4+c_0$. The normalization $|\nabla f|^2=f$ forces $c_0=0$.
Finally, Wylie's theorem gives finiteness of the fundamental group of the
original shrinker. Thus the original manifold is a finite quotient of
$S^2(\sqrt2)\times\mathbb R^3$, as asserted.
\end{proof}

\end{document}